\documentclass[11pt]{article}
\usepackage{mathrsfs}
\usepackage{amsmath, mathtools}
\usepackage{amssymb}
\usepackage{amsthm}
\usepackage{graphicx}
\usepackage{epic}
\usepackage{xcolor,cite}
\usepackage{cases}
\usepackage{pst-poly}  
\usepackage{pst-plot}  
\usepackage{pst-poly}  
\usepackage{CJK,makecell}

\usepackage{verbatim}
\usepackage{subfig}
\usepackage{booktabs}
\definecolor{dkgreen}{rgb}{0,0.6,0}
\numberwithin{equation}{section}

\usepackage{etex}

\usepackage[left,mathlines,displaymath]{lineno}

\usepackage{pgf}
\renewcommand{\paragraph}{\roman{paragraph}}

\usepackage{bm}
\usepackage[hidelinks]{hyperref}
\usepackage{tikz}
\usetikzlibrary{automata}
\usepackage{enumitem}

\usetikzlibrary{arrows,shapes,positioning}
\usetikzlibrary{decorations.markings}
\tikzstyle arrowstyle=[scale=1]
\tikzstyle directed=[postaction={decorate,decoration={markings, mark=at position .65 with {\arrow[arrowstyle]{stealth}}}}]
\tikzstyle reverse directed=[postaction={decorate,decoration={markings, mark=at position .65 with {\arrowreversed[arrowstyle]{stealth};}}}]

\newtheorem{claim}{Claim}[section]
\newtheorem{theorem}{Theorem}[section]
\newtheorem{corollary}[theorem]{Corollary}

\newtheorem{lemma}[theorem]{Lemma}

\newtheorem{proposition}[theorem]{Proposition}
\newtheorem{Remark}[theorem]{Remark}

\begin{document}

\title{On asymptotic values for the  minimum number of spanning forests in simple regular graphs\thanks{$\dag$ Corresponding author. E-mail addresses: xush0928@163.com (S.\ Xu), kexxu1221@126.com
(K.\ Xu).}}\author{{Shaohan Xu $^{a}$, Kexiang Xu $^{a,\dag}$}\\\\{\small $^{a}$ School of Mathematics, Nanjing University of Aeronautics and Astronautics,}\\
{\small Nanjing, Jiangsu 211016, PR China}\\
}

\maketitle
\begin{abstract}
Let $F(G)$ be the number of spanning forests in a graph $G$ and $\mathcal{C}(n,d)$ be the set of all connected $d$-regular simple graphs of order $n$. Define
$\widehat{f}_{d}=\liminf_{n\rightarrow \infty}\{F(G)^{1/n}:G\in \mathcal{C}(n,d)\}$. Let  $n_i$  be the number of vertices of degree $i$ in $G$. In this paper we give two lower bounds for $F(G)$ in terms of $n_i$ in  connected graphs whose vertex degrees belong to $\{2,3\}$ and $\{2,3,4\}$, respectively. Furthermore, we determine the exact values of $\widehat{f}_3$ and $\widehat{f}_4$.
\end{abstract}

\section{Introduction}
All graphs considered in this paper are undirected and  loopless. By a graph we mean a  simple graph,  while a multigraph refers to a graph that may  have multiple edges.  Let $G$ be a graph with vertex set $V(G)$ and edge set $E(G)$. A {\it spanning forest} in the graph $G$ is a spanning subgraph of $G$  which is a disjoint union of trees. Here an isolated vertex is considered as a tree.   Denote  by $F(G)$  the number of spanning forests in $G$ and use $\tau(G)$ to denote the number of spanning trees of $G$.

A graph is  {\it $d$-regular} if each vertex has degree $d$. Let $\mathcal{C}(n,d)$ be the set of all connected $d$-regular simple graphs of order $n$. In $1983$, McKay \cite{M1} determined precisely the asymptotic value of the maximum number of spanning trees of graphs in $\mathcal{C}(n,d)$, by proving that
$$\limsup_{n\rightarrow \infty}\{\tau(G)^{1/n}:G\in \mathcal{C}(n,d)\}=\frac{(d-1)^{d-1}}{(d^{2}-2d)^{d/2-1}},$$
and the best-known asymptotic formula  was established by Chung and Yau \cite{C2} in $1999$. By contrast, the situation is less understood for the analogous inferior limit. Let $c(d)=\liminf_{n\rightarrow \infty}\{\tau(G)^{1/n}:G\in \mathcal{C}(n,d)\}$. In $1990$, Alon \cite{A1} proved the following inequalities:
$$d-\Theta\left(d(\log\log d)^{2}/\log d\right)\leq c(d)\leq \left[(d+1)^{d-2}(d-1)\right]^{1/(d+1)},$$
where $d\geq3$ for the right inequality.  Alon also derived many other properties of $c(d)$ in his work, and  he proposed to determine the exact value of $c(d)$ for each $d\geq3$. Denote by $K_{n}$ the complete graph of order $n$. For even number $n$, let $K_n^{-}$  be the graph obtained from $K_n$ by deleting  an arbitrary perfect matching of $K_n$. In $1995$, Kostochka \cite{K1} improved the asymptotic range for $c(d)$ and generalized it to graphs with a given degree sequence, by subtly modifying Alon's approaches.  Furthermore,  Kostochka \cite{K1}  showed that $c(3)=2^{3/4}$ by proving that $\tau(G)\geq2^{3(n_3+2)/4}$ for a connected graph $G\neq K_4$ with vertex degrees  from $\{2,3\}$. In $2024$,  Sereni and  Yilma \cite{S1} determined that $c(4)=75^{1/5}$ by proving that   $\tau(G)\geq75^{n_4/5+n_3/10+1/5}$ for a connected graph $G \not\in\{K_5,K_6^{-}\}$ with vertex degrees  from $\{2,3,4\}$. To our best knowledge, the exact value of $c(d)$ is not yet determined for $d\geq 5$  as a difficult problem underlined by Alon \cite{A1}. However, the analogous questions for connected regular multigraphs have been completely solved in \cite{B2,P1} by characterizing the extremal multigraphs with the minimum number of spanning trees.

One  expects similar results to be true for $F(G)$ as $\tau(G)$, but it turns out that they are different in nature. It is known that $\tau(G)$ can be computed in polynomial time by the matrix-tree theorem \cite{K2}. However, computing $F(G)$ is $\#P$-hard \cite{J1}. Nevertheless one expects similar or at least analogous behaviors in certain problems even if they cannot be proven for spanning forests.

In $1996$,  Kahale and Schulman \cite{K3} established an upper bound on $F(G)$, that is,
\begin{equation}\label{x09}
F(G)^{1/n}\leq\frac{d+1}{\eta}\left(\frac{d-1}{d-\eta}\right)^{(d-2)/2}=d+\frac{1}{2}+O\left(\frac{1}{d}\right),
\end{equation}
where $G\in\mathcal{C}(n,d)$ and $\eta=\frac{(d+1)^{2}-(d+1)(d^{2}-2d+5)^{1/2}}{2(d-1)}$. Let $g(G)$ be the girth of $G$, that is, the length of the shortest cycle in $G$.  In $2021$, Borb\'enyi, Csikv\'ari and  Luo \cite{B3} studied the values of $f_d=\limsup_{n\rightarrow \infty}\{F(G)^{1/n}:G\in \mathcal{C}(n,d)\}$ and  determined that $f_3=2^{3/2}$  with a sequence $\{G_n\}_{n=1}^{\infty}$ of $3$-regular graphs satisfying $g(G_n)\rightarrow\infty$, at which the value of $f_3$ is realized. Moreover, they improved on the previous upper bounds in  \eqref{x09}  on $F(G)$ for $4\leq d\leq 9$. In $2022$,  Bencs and  Csikv\'ari \cite{B1} proved the following result: if $\{G_n\}_{n=1}^{\infty}$ is a sequence  of $d$-regular graphs with $d\geq 3$ and $g(G_n)\rightarrow\infty$, then
\begin{displaymath}
\lim_{n\rightarrow\infty}F(G_n)^{1/n}=\frac{(d-1)^{d-1}}{(d^{2}-2d-1)^{d/2-1}},
\end{displaymath}
which extended the case $d=3$ in \cite{B3} to the general case. In $2023$, Bencs and  Csikv\'ari \cite{B4} further showed that $F(G)\leq d^{n}$ for any graph $G\in \mathcal{C}(n,d)$, which  improved all the  aforementioned upper bounds on $F(G)$.

In this paper we mainly study the lower bounds on $F(G)$. Let
$\widehat{f}_{d}=\liminf_{n\rightarrow \infty}\{F(G)^{1/n}:G\in \mathcal{C}(n,d)\}$ for $d\geq 3$. Inspired by these aforementioned results,  we get the lower bounds for $F(G)$ in terms of the number of vertices of degree $i$ in $G$. Furthermore, we  determine the exact values of $\widehat{f}_{d}$ for $d=3$ and $4$, respectively.  Let $n_{i}(G)$ or $n_i$ for short, be the number of vertices of degree $i$ in $G$.  More precisely, we state our main results  as follows.

\begin{theorem}\label{x2}
 Let $G$ be a connected graph whose vertex degrees are from $\{2,3\}$, then
 $$F(G)\geq2^{n_2+n_3-1}3^{\frac{n_3+2}{4}},$$
 unless $G\cong K_{4}$.
 \end{theorem}

\begin{theorem}\label{xh-1}
 Let $G$ be a connected graph whose vertex degrees are from $\{2,3,4\}$, then
 $$F(G)\geq 2^{n_2+\frac{3n_3+n_4-9}{5}}198^{\frac{n_3+2n_4+2}{10}},$$
 unless $G\in \mathcal{E}=\{K_5,K_6^{-}\}$.
 \end{theorem}

By Proposition \ref{x1111}, we get a direct consequence  of Theorems \ref{x2} and \ref{xh-1} as follows.
\begin{corollary}\label{xx1}
Let $\widehat{f}_{d}$ be a number defined as above. Then $\widehat{f}_{3}=2\cdot 3^{1/4}$ and $\widehat{f}_{4}=2^{2/5}\cdot99^{1/5}$.
 \end{corollary}

\section{Preliminaries}

\subsection{Notations and techniques}
For a vertex $v\in V(G)$, we denote by $N_{G}(v)$, or simply $N(v)$, the set of vertices adjacent to $v$ in a graph $G$.  The {\it degree} of  $v$ in $G$ is the number of edges  containing  $v$, denoted by $d_{G}(v)$, or simply $d(v)$.  For any $V'\subseteq V(G)$ and $E'\subseteq E(G)$, let $G[V']$ (resp.\ $G[E']$) be the subgraph of $G$ induced by $V'$ (resp.\ $E'$). We write $G-V'$ to mean $G[V(G)\setminus V']$, and abbreviate $G-\{v\}$ to $G-v$ for a vertex $v\in V(G)$. For any $E'\subseteq E(G)$, let $G-E'$ be the subgraph of $G$ obtained by deleting  all edges in $E'$, and abbreviate $G-\{e\}$ to $G-e$ for an edge $e\in E(G)$. For any positive integer $n$, we write the set $\{1, 2, \ldots, n\}$ as $[n]$.

For a graph $G=(V,E)$, let $T_G(x,y)$ denote its Tutte polynomial \cite{T2}, that is,
\begin{displaymath}
T_G(x,y)=\sum_{A\subseteq E}(x-1)^{k(A)-k(E)}(y-1)^{k(A)+|A|-|V|},
\end{displaymath}
where $k(A)$ denotes the number of connected components of the graph $(V,A)$.
The Tutte polynomial encodes a lot of quantitative information about the graph $G$, including $F(G)$ and $\tau(G)$. See some relevant results in \cite{D1,E1,H1,T3}.

\begin{lemma}[\hspace{1sp}{\cite{T2}}]\label{x133}
 For a multigraph $G$, we have $\tau(G) = T_G(1, 1)$ and $F(G)=T_G(2, 1)$.
 \end{lemma}

For a subset $V'\subseteq V(G)$,  a {\it contraction} on $V'$ in a graph $G$ is to  identify all the  vertices in $V'$, and always deleting possible loops and retaining possible multiple edges, unless explicitly stated. The resulting graph is denoted   by  $G/V'$.   In particular, for $V'=\{u,v\}$ with  $uv\in E(G)$, we write  $G/\{u,v\}$  simply as $G/ uv$. The following  two lemmas are similar to the number of  spanning trees in  a graph.

\begin{lemma}[\hspace{1sp}{\cite{B3}}]\label{x1}
 For a graph $G$ with $e\in E(G)$, we have $F(G)=F(G-e)+F(G/e)$.
 \end{lemma}

\begin{lemma}\label{x111}
 Let $G$ be a graph and $v$ be a cut-vertex of $G$. Assume that  $G_{1},G_2,...,G_t$ are all connected  components of $G-v$. Then $F(G)=\prod_{i=1}^{t}F(G[V(G_i)\cup \{v\}]).$
 \end{lemma}

\begin{lemma}\label{x11}
 Let $G$ be a graph and $V'\subseteq V(G)$, then the number of spanning forests   in $G$ where  all vertices in $V'$ are in different components  is equal to $F(G/V')$. Furthermore, $F(G/V')\leq F(G/V'')$ if $V''\subseteq V'\subseteq V(G)$.
 \end{lemma}
\begin{proof}
The second part is obvious.  Here we mainly prove the first part. Let $\mathcal{F}_{diff}(G,V')$ be the set of spanning forests   in $G$ where all vertices in $V'$  belong to different components, and let $\mathcal{F}(G/V')$ be the set of spanning forests   in  $G/V'$.   It suffices to prove that there is a bijection between  $\mathcal{F}_{diff}(G,V')$ and $\mathcal{F}(G/V')$.

We define a mapping $\phi: \mathcal{F}_{diff}(G,V')\rightarrow \mathcal{F}(G/V')$, that is,   by  contracting  $V'$ into a new vertex $v^*$,   transforming $T_0$ into a spanning subgraph $T_0'$ of $G/V'$ for any $T_0\in \mathcal{F}_{diff}(G,V')$. Then $T_0'\in\mathcal{F}(G/V')$   since  all vertices in $V'$ are in different components of $T_0$. Clearly, this mapping   is well-defined. Conversely, for any  $T_0'\in\mathcal{F}(G/V')$, we derive a  spanning forest $T_0\in\mathcal{F}_{diff}(G,V')$ by the following transformation:  expanding the vertex $v^{*}$ in $T_0'$ into the original vertex set $V'$ while adding no edges  in $V'$, and then replacing each edge  $e=v^{*}u$  in $T_0'$ with an edge $v_i u$   if $e=v^{*}u$  is  obtained from $v_i u$ after identifying  $V'$. For any $T_0\in \mathcal{F}_{diff}(G,V')$, there exists a unique $T_0'\in \mathcal{F}(G/V')$  corresponding to $T_0$.  Thus,  this mapping is also a surjection.

On the other hand, for any two distinct  spanning forests $F_1,F_{2}\in \mathcal{F}_{diff}(G,V')$, there  exist  two   distinct edges  $e_1\in E(F_1)\setminus E(F_2)$ and $e_2\in E(F_{2})\setminus E(F_1)$. Note that $e_1,e_2\not\in G[V']$.  By the definition of the contracting operation,   there  exist  two distinct edges corresponding to $e_1$ and $e_2$ in $\phi(F_1)$ and $\phi(F_2)$, respectively, regardless of whether $e_1$ and $e_2$  incident with the vertices in $V'$ or not. Then $\phi(F_1)\neq \phi(F_2)$.  Therefore,  $\phi$ is an injection.
 {\hfill}
\end{proof}

\begin{Remark}\label{xx234}
   Let $T_0$ be a forest  in $G$ with components $T_{1},T_{2},\ldots,T_{c}$ and $G_1$ be a subgraph of $ G-E(T_0)$ with $V(G)\setminus V(T_0)\subseteq V(G_1)$. Let $F_{G_1}(G,T_0)$ be the number of spanning forests of $G$ each of which is obtained by adding to $T_0$ possible edges  in $G_1$. Note that such a  spanning forest contains no cycle.  Then    $F_{G_1}(G,T_0)=F(G^{*})$ by Lemma \ref{x11}, where $G^*$ is obtained from  $G_1$ by contracting each vertex set  $V(T_i)\cap V(G_1)$  for $i\in [c]$ into a single vertex, respectively.
\end{Remark}

 Let $G$ and $G'$ be two  graphs with $V_0\subseteq V(G),E_0\subseteq E(G)$ and $V_0'\subseteq V(G'),E_0'\subseteq E(G')$.
Assume that they have a common subgraph $H=G-V_0-E_0=G'-V_0'-E_0'$. Let $E_1$ and $E_1'$ be the set of edges incident with $V_0$ and $V_0'$, respectively. Suppose that $G_1=G[E_0\cup E_1]$ and $G_1'=G'[E_0' \cup E_1']$.
 Combining  Lemma \ref{x11} with Remark \ref{xx234}, we  get a lower bound for the ratio $F(G)/F(G')$ by comparing the number of ways of extending any spanning forest of $H$ to a spanning forest of $G$ and $G'$ by using only edges in $G_1$ and $G_1'$, respectively.

\begin{lemma}\label{x876}
Keeping the above notations, we have
 $$\frac{F(G)}{F(G')}\geq \min_{T_0}\frac{F_{G_1}(G,T_0)}{F_{G_1'}(G',T_0)},$$
 where the minimum is taken over all spanning forests $T_0$ of $H$.
 \end{lemma}
\begin{proof}
Let $\mathcal{F}(H)$ be the set of spanning forests of $H$. Since every
edge of $G-E(H)$ lies in $G_1$, and every edge of $G'-E(H)$ lies in $G'_1$, every
spanning forest of $G$ or $G'$ is obtained uniquely by starting with a
spanning forest $T_0$ of $H$ and adding some possible edges from $G_1$ or
$G'_1$, respectively. Hence,
\[
F(G)=\sum_{T_0\in \mathcal{F}(H)}F_{G_1}(G,T_0)
\quad\text{and}\quad
F(G')=\sum_{T_0\in \mathcal{F}(H)}F_{G'_1}(G',T_0).
\]
Therefore,
\begin{displaymath}
\begin{split}
F(G)
&=\sum_{T_0\in \mathcal{F}(H)}F_{G_1}(G,T_0)\\
&\ge
\left(
\min_{T_0\in \mathcal{F}(H)}
\frac{F_{G_1}(G,T_0)}{F_{G'_1}(G',T_0)}
\right)
\sum_{T_0\in \mathcal{F}(H)}F_{G'_1}(G',T_0)\\
&=\left(
\min_{T_0\in \mathcal{F}(H)}
\frac{F_{G_1}(G,T_0)}{F_{G'_1}(G',T_0)}
\right)F(G'),
\end{split}
\end{displaymath}
which implies the result.
{\hfill}
\end{proof}

In \cite {O1}, Ok and Thomassen  introduced a ``lifting" operation  of multigraphs to study the minimum number of spanning trees in multigraphs. Let $f_1=xy_1$ and $f_2=xy_2$ with $y_1\neq y_2$ be two adjacent edges  in a  multigraph $G$. {\it Lifting} $f_1$ and $f_2$ means deleting  $f_1$ and $f_2$ and adding a new edge between $y_1$ and $y_2$. If $x$ is a vertex of degree $2m$ in $G$, a {\it complete lift} of $x$ is the process of first performing a sequence of $m$ lifts of pairs of edges incident with $x$ and then deleting the vertex $x$ (which is, by then, isolated), thereby producing a multigraph  $G_x$. Ok and Thomassen \cite{O1} found a link between the numbers of spanning trees in $G$ and $G_x$, and Pek\'arek, Sereni and  Yilma  refined this result, see \cite{P1} for details.

\begin{theorem}[\hspace{1sp}{\cite{O1,P1}}]\label{x311}
 Let $G$ be a multigraph with a vertex $x$ of degree $2m$. Let $G_x$ be a multigraph obtained from $G$ by a complete lift of $x$. Then
 \begin{displaymath}
\tau(G)\geq c_{m}\tau(G_x),
\end{displaymath}
with
\begin{displaymath}
c_m=\min_{d_1,d_2,...,d_k}\min_{X}\frac{\prod_{i=1}^{k}d_i}{\tau(X)}
\end{displaymath}
where the minimum is taken over all sequences of positive integers $d_1,d_2,\ldots, d_k$ with varying length $k$ such that $\sum_{i=1}^{k}d_i=2m$, and over all connected multigraphs $X$ of order $k$
with degree sequence $d_1,d_2,\ldots, d_k$.
 \end{theorem}

Given a multigraph $G$ and any two vertices $i,j\in V(G) $, let $\omega_{G}(i,j)$  be the number of edges of $G$ between $i$ and $j$.  Pek\'arek, Sereni and  Yilma \cite{P1} determined when we can indeed perform a complete lift of the vertex $x$ with the resulting connected multigraph $G_x$.

\begin{lemma}[\hspace{1sp}{\cite{P1}}]\label{x354}
 Let $G$ be a connected multigraph with a vertex $x$ of degree $2m$. There exists  a complete lift of $x$ yielding a connected multigraph $G_x$  if and only if $\omega_G(x,y)\leq m$ for every vertex $y$ of $G$ and $G-x$ has at most $m+1$ components.
 \end{lemma}

Using the  methods in \cite{O1} and \cite{P1}, we can extend Theorem \ref{x311} to the number of spanning forests in a graph in the following.

\begin{theorem}\label{x322}
 Let $G$ be a multigraph with a vertex $x$ of degree $2m$. Let $G_x$ be a multigraph obtained from $G$ by a complete lift of $x$. Then
 \begin{displaymath}
F(G)\geq \ell_{m}F(G_x),
\end{displaymath}
with
\begin{displaymath}
\ell_m=\min_{d_1,d_2,...,d_k}\min_{X}\frac{\prod_{i=1}^{k}(d_i+1)}{F(X)},
\end{displaymath}
where the minimum is taken over all sequences of positive integers $d_1,d_2,\ldots, d_k$ with varying length $k$ such that $\sum_{i=1}^{k}d_i=2m$, and over all connected multigraphs $X$ of order $k$ with degree sequence
 $d_1,d_2,\ldots, d_k$.
 \end{theorem}
\begin{proof}
Let $G_0=G-x$ and $N=\{x_i:xx_i\in E(G)~\mathrm{and}~i\in[2m]\}$ be the multi-set of all  neighbours of $x$ in $G$.  Assume that we can perform  a complete lift of the vertex $x$ on $G$.
Let $G_1$ be a subgraph induced by the edge set $\{xx_i:i\in[2m]\}$.  Note that  $E(G_x)-E(G_0)$ is the set of edges created by lifting $x$, which induces a subgraph $G_2$.  Observe that  $G_0=G-x=G_x-E(G_2)$. Now we consider any spanning forest $T_0$ of $G_0$ that can be extended to a spanning forest of $G_x$ by adding (possibly $0$) edges in $E(G_2)$. We shall compute $F_{G_1}(G,T_0)$ and $F_{G_2}(G_x,T_0)$.

Suppose that $T_1$, $T_2$,\ldots,$T_k$ with orders $|T_i\cap N|=d_i\geq1$  are all components of $T_0$ which contain the neighbours of $x$, where   $\sum_{i=1}^{k}d_i=2m$.  By Remark \ref{xx234}, $F_{G_1}(G,T_0)=F(S^{*})=\prod_{i=1}^{k}(d_i+1)$, where $S^*$ is a  multi-star at $x$ with edge multiplicities $d_1$, $d_2$,\ldots,$d_k$. Likewise, by Remark \ref{xx234}, we have $F_{G_2}(G_x,T_0)=F(H)$, where $H$ is the graph obtained from $G_2$ by contracting each vertex set $V(G_2)\cap T_i$ for $i\in [k]$ into  a single vertex. Let $v_i$ be the vertex of $H$ corresponding to $V(G_2)\cap T_i$. Then $d_{H}(v_i)\leq d_i$  since each vertex $x_j\in T_i$ provides $v_i$ with at most one edge from $E(G_2)$. Moreover, by Lemma \ref{x876} and considering all possibilities for $T_0$, we get
\begin{equation*}
\frac{F(G)}{F(G_x)}\geq\min_{T_0}\frac{F_{G_1}(G,T_0)}{F_{G_2}(G_x,T_0)}=\min_{d_1,d_2,...,d_k}\min_{H}\frac{\prod_{i=1}^{k}(d_i+1)}{F(H)}\triangleq a_{m}.
\end{equation*}

We consider the sequence  $D=\{d_1,d_2,\ldots, d_k\}$ of positive integers with length $k$.  Assume that $H$ has degree sequence $D'=\{d_1',d_2',\ldots,d_k'\}$, where $d_i'=d_{H}(v_i)\leq d_i$ for $i\in[k]$. By the definition of  $\ell_m$, it suffices to prove that   $a_{m}$   takes the minimum value only if  $H$ is  connected and   the sequence $D'=D$, i.e., $H$ is isomorphic to $X$ and $a_m=\ell_m$.

Suppose that  $a_{m}$   takes the minimum value and $H$ is not connected. Let  $H_1,H_2,\ldots, H_c$  be all components with $c\geq2$ in $H$. Without loss of generality, assume that  $v_i\in H_i$  for $i\in [c]$.   We construct a new graph $H'$ obtained from $H$ by identifying the vertex set $\{v_1,v_2,\ldots,v_c\}$ into a new vertex $v^*$ of degree $\sum_{i=1}^{c}d_i'$. Then $H'$ has $k-c+1$ vertices with degree sequence  $\sum_{i=1}^{c}d_i',d_{c+1}',d_{c+2}',\ldots,d_k'$. Consider a sequence $D''=\{\sum_{i=1}^{c}d_i,d_{c+1},d_{c+2},\ldots,d_k\}$. Then
$$\left(1+\sum_{i=1}^{c}d_i\right)\cdot\prod_{i=c+1}^{k}(d_i+1)< \prod_{i=1}^{k}(d_i+1)$$
as $d_i\geq1$ for each $i\in[k]$. Moreover, $F(H')=F(H)$. Hence,  $a_{m}$  takes the smaller value on a connected graph $H'$ along with the sequence $D''$, a contradiction. Hence, we may assume that $H$ is connected, which also implies that $d_i'\geq1$.

 Now our goal is to show that if $a_{m}$   takes the minimum value, then $D'=D$. First, if there are at least two vertices  $v_i\neq v_j$ such that $d_i'< d_i$ and $d_j'< d_j$, we can add a new edge between $v_i$ and $v_j$ in $H$ to form a new connected graph $H''$. Obviously, $F(H'')> F(H)$. Hence, $H''$ along with the sequence $D$ implies that $a_m$ cannot attain the minimum value. We thus assume that there exists a unique $v_i$ for $i\in[k]$ such that $1\leq d_i'<d_i$. It follows that $d_i'\leq d_i-2$ since both $\sum_{i=1}^{k}d_i$ and $\sum_{i=1}^{k}d_i'$ are even.

 Consider a sequence $D'''=\{s_j\}_{j=1}^{k+1}$ defined by
 \begin{displaymath}
  s_j= \begin{cases}
d_j &\mbox{if  $j\neq i$},\\
d_i-1 &\mbox{if $j=i$},\\
1 &\mbox{if $j=k+1$},\\
 \end{cases}
\end{displaymath}
which satisfies that $\sum_{j=1}^{k+1}s_j=\sum_{j=1}^{k}d_j=2m$. Note that
$$\prod_{j=1}^{k+1}(s_j+1)=2\left(1-\frac{1}{d_i+1}\right)\cdot\prod_{j=1}^{k}(d_j+1) <2 \prod_{j=1}^{k}(d_j+1).$$
Let $H'''$ be a connected graph obtained from $H$ by adding a new vertex of degree $1$ joined to $v_i$. Then $H'''$ has degree sequence  $s_1',s_2',\ldots, s_{k+1}'$, where $s_j'=d_j'$ if $j\neq i$, while $s_i'=d_i'+1$ and $s_{k+1}'=1$. Observe that $s_j'\leq s_j$ for each $j\in [k+1]$. Moreover, $F(H''')=2F(H)$. Therefore, $H'''$ along with the sequence $D'''$ implies that $a_m$ cannot attain the minimum value, which concludes the proof.
{\hfill}
\end{proof}

In  subsequent work we restrict our attention to simple graphs.  The  lifting operation is applied only to pairs of edges  $f_1=xy_1$ and $f_2=xy_2$ in a graph $G$ with $y_1y_2\not\in E(G)$, thereby producing a simple graph  $G_x$. Note that Theorem \ref{x322}  also holds for simple graphs.    With our definition, it is possible to perform a complete lift at $x$ of degree $2m$  only if there exists a matching of size $m$ in the complement of  $G[N_G(x)]$. In the above definition of $\ell_m$, the connected multigraph $X$ has exactly $m$ edges, so $\ell_1=2$. Furthermore, $\ell_2=3$ and $\ell_3=\frac{27}{7}$, which are attained by a  $2$-cycle and $3$-cycle, respectively.

\subsection{Upper bound on $\widehat{f}_{d}$}

Recall that $\widehat{f}_{d}=\liminf_{n\rightarrow \infty}\{F(G)^{1/n}:G\in \mathcal{C}(n,d)\}$ for $d\geq 3$.  Now we  show that the following fact about $\widehat{f}_d$: for any fixed $d$-regular connected graph $G$ of order $r$ and every non-cut edge $uv$ of $G$,
\begin{equation}\label{2.1}
\widehat{f}_d\leq[2F(G-uv)]^{\frac{1}{r}}.
\end{equation}
To see this, we consider an arbitrary  $d$-regular connected graph $G$ on $r$ vertices, where  $uv$ is a non-cut edge of $G$ (there is always such an edge as $G$ contains a cycle). For each integer $m\geq 2$, let $H_0,H_1,\ldots, H_{m-1}$ be $m$ copies of $G-uv$ and let $u_i$ and $v_i$ be the vertices of $H_i$ corresponding to $u$ and $v$, respectively, where $ 0\leq i<m$. Let $G_m$ be the graph obtained from the disjoint union of $H_0,H_1,\ldots, H_{m-1}$ by adding to it the edges $v_iu_{i+1}$ ($ 0\leq i<m$), where the indices are reduced mod $m$. Obviously $G_m$ is a $d$-regular connected graph on $mr$ vertices. By  a routine computation, we have
\begin{displaymath}
F(G_m)=2^{m}[F(G-uv)]^{m}-[F(G-uv)-F(G/uv)]^{m},
\end{displaymath}
where $F(G-uv)-F(G/uv)$ means the number of spanning forests in $G-uv$ such that $u$ and $v$ are in a same component, by using Lemma \ref{x11}.
Note that $\lim_{n\rightarrow \infty}(a^{n}-b^{n})^{1/n}=a$ if $a>b>0$ and $ F(G)>F(G-uv)>F(G-uv)-F(G/uv)$.
Then $\lim_{m\rightarrow \infty}F(G_m)^{1/mr}=[2F(G-uv)]^{1/r}$. Moreover, by the definition of  $\widehat{f}_d$, we have
 $$\widehat{f}_d\leq \lim_{m\rightarrow \infty}F(G_m)^{\frac{1}{mr}} =[2F(G-uv)]^{\frac{1}{r}},$$
 and \eqref{2.1} follows.

 In this paper  we compute the number of spanning forests in some small graphs, which can be obtained by computing the evaluations of $T_G(2,1)$ in the Tutte polynomial via  version $9.3$ of the SageMath software system \cite{S2}. We have included a link to SageMath code used to verify these calculations in the appendix.  Based on the above observation, we get the following statement.
\begin{proposition}\label{x1111}
For each $d\geq 3$, $\widehat{f}_d\leq[2F(K_{d+1}-e)]^{1/(d+1)}$. In particular, $\widehat{f}_3\leq 2\cdot3^{1/4}$ and $\widehat{f}_4\leq 2^{2/5}\cdot99^{1/5}$.
 \end{proposition}
\begin{proof}
Taking $G$ in \eqref{2.1} as $K_{d+1}$,   we have
$F(K_4-e)=24$ and $F(K_5-e)=198$ by employing SageMath, the result follows.
{\hfill}
\end{proof}

\section{Proof of Theorem \ref{x2}}
In the present section we prove Theorem \ref{x2}. Before that, we need to do some preparatory work.  Let $K_{a,b}$ be a   complete bipartite graph with  bipartition $X\cup Y$, where $|X|=a$ and $|Y|=b$. Denote by $R_1$ the $3$-prism, which is the Cartesian product of a $3$-cycle $C_3$ and $K_2$. Let $R_2$ be a graph obtained from $K_4$ by subdividing an edge of $K_4$. For any graph $G$, let
$$p(G)=2^{n_2(G)+n_3(G)-1}3^{\frac{n_3(G)+2}{4}}.$$
 Now we compute $F(G)$ and $p(G)$ for several graphs $G$ by using SageMath as follows.

\begin{lemma}\label{x73}
$F(K_3)=7>2^{2}3^{1/2}=p(K_3)$, $F(K_4)=38<2^{3}3^{3/2}=p(K_4)$, $F(K_4-e)=24=p(K_4-e)$ for any $e\in E(K_4)$, $F(K_{3,3})=328>288=p(K_{3,3})$, $F(R_{1})=314>288=p(R_1)$ and $F(R_{2})=86>2^{4}3^{3/2}=p(R_2)$.
 \end{lemma}

 Let $\mathcal{W}$ be the set of all graphs except $K_4$ in Lemma \ref{x73}. By Lemma \ref{x73}, Theorem \ref{x2} is true for any graph $G\in \mathcal{W}$. For each $n\geq 3$, let $\mathcal{A}_{n}$ be the class of all connected graphs of order $n$, whose vertex degrees are from $\{2,3\}$. Suppose that $G\in \mathcal{A}_{n}\setminus \{K_4\}$ is a  graph with the minimum number of edges  such that $F(G)<p(G)$. Then $G\not\in \mathcal{W}$. Note that  $|E(G)|\geq4$ since $G\neq K_3$. Our aim is to derive a contradiction, which will establish this theorem. Next we prove  some claims in the following.

\begin{claim}\label{x-2}
$G$ has no cut-edge.
\end{claim}
\begin{proof}[\em\textbf {Proof of Claim \ref{x-2}}]
 Suppose on the contrary that $e$ is a cut-edge of $G$. There exists a path $P=v_{1}v_{2}\cdots v_s$ of length $s-1$ in $G$ that contains $e$, where $d(v_1)=d(v_s)=3$, $d(v_2)=d(v_3)=\cdots=d(v_{s-1})=2$ and  all the edges of $P$ are cut-edges. Deleting the interior vertices of $P$ forms two connected components $G_1$ and $G_2$. Both $G_1$ and $G_2$  have a vertex of degree $2$ (namely, $v_1$ and $v_s$), and   hence each of them belongs to some $ \mathcal{A}_{k}\setminus \{K_4\}$ where $k\geq 3$. Since $n_{2}(G_1)+n_{2}(G_2)=n_{2}(G)-s+4$  and $n_3(G_1)+n_3(G_2)=n_3(G)-2$, we have $p(G_1)p(G_2)=2^{-(s-1)}p(G)$.  By the minimality of $G$, we have
$$F(G)=2^{s-1}F(G_1)F(G_2)\geq 2^{s-1}p(G_1)p(G_2)=p(G).$$
This contradiction concludes Claim \ref{x-2}.
{\hfill}
\end{proof}

\begin{claim}\label{x-3}
$G$ is $3$-regular.
\end{claim}
\begin{proof}[\em\textbf {Proof of Claim \ref{x-3}}]
Suppose that $G$ contains a vertex $v$ with $N_G(v)=\{x,y\}$.

\medskip
{\indent\bf Case 1.} $xy\not\in E(G)$. In this case let $G'=(G-v)\cup \{xy\}$. Observe that $n_{2}(G')=n_{2}(G)-1$ and $n_3(G')=n_3(G)$. Then $p(G')=2^{-1}p(G)$. If $G'=K_4$, then $G=R_2\in \mathcal{W}$, a contradiction. Otherwise by the minimality of $G$ and Theorem \ref{x322},
 $F(G)\geq2F(G')\geq2p(G')
 =p(G)$,
a contradiction.

\medskip
{\indent\bf Case 2.} $xy\in E(G)$. Since $G\neq K_3$ and $G$ has no cut-edge by Claim \ref{x-2}, we can infer that $d_G(x)=d_G(y)=3$.  Suppose that $N_G(x)=\{v,y,x'\}$ and $N_G(y)=\{v,x,y'\}$. If $x'=y'$, then  either $G=K_4-e\in \mathcal{W}$, or the third edge incident with $x'=y'$  is a cut-edge, both are contradictions. Thus, we assume that $x'\neq y'$. Now we form $G'$ from $G$ by contracting the vertices $x,y,v$ into a new vertex $z$ of degree $2$, and obviously $G'\in \mathcal{A}_{n-2}\setminus \{K_4\}$. Since $n_{2}(G')=n_{2}(G)$ and $n_3(G')=n_3(G)-2$, we have $p(G')=2^{-2}3^{-1/2}p(G)$. Note that each spanning forest of $G'$ can be extended to a spanning forest of $G$ by at least $F(K_3)=7$ ways obtained by adding any  spanning forest of $G[\{x,y,v\}]\cong K_3$.  By the minimality of $G$, $$F(G)\geq F(K_3)F(G')\geq7p(G')=7\cdot2^{-2}3^{-1/2}p(G)>p(G),$$
a contradiction. Therefore, Cases $1$ and $2$ conclude Claim \ref{x-3}.
{\hfill}
\end{proof}

\begin{claim}\label{x-4}
$G$ contains no subgraph isomorphic to $K_4-e$.
\end{claim}
\begin{proof}[\em\textbf {Proof of Claim \ref{x-4}}]
 Suppose that $G$ contains an induced subgraph $K_4-e$ with  vertex set $V'=\{u,v,x,y\}$ and $xy\not\in E(G)$. Let $x'$ and $y'$ be the neighbours of $x$ and $y$ not in $V'$, respectively. Since $G$ is $3$-regular and has no cut-edge, we have $x'\neq y'$.  Now we form $G'$ from $G$ by contracting  $V'$ into a new vertex $z$ of degree $2$, and obviously $G'\in \mathcal{A}_{n-3}\setminus \{K_4\}$. Since $n_{2}(G')=n_{2}(G)+1$ and $n_3(G')=n_3(G)-4$, we have $p(G')=2^{-3}3^{-1}p(G)$. Note that each spanning forest of $G'$ can be extended to a spanning forest of $G$ by at least $F(K_4-e)=24$ ways obtained by adding any  spanning forest of $G[V']\cong K_4-e$.  By the minimality of $G$, we have
 $$F(G)\geq 24F(G')\geq24p(G')=24\cdot2^{-3}3^{-1}p(G)=p(G),$$ a contradiction, the result follows.
{\hfill}
\end{proof}

\begin{claim}\label{x-5}
$G$ contains no triangles.
\end{claim}
\begin{proof}[\em\textbf {Proof of Claim \ref{x-5}}]  Suppose that  $V'=\{x,y,z\}$ induces a triangle in $G$.  Let $x'$, $y'$ and $z'$ be the neighbours of $x$, $y$ and $z$ not in $V'$, respectively. Due to Claim \ref{x-4}, we have  $x'$, $y'$ and $z'$ are pairwise distinct.  Let $G'$ be the graph obtained from $G$ by contracting  $V'$ into a new vertex $w$ of degree $3$.   If $G'=K_4$, then $G=R_1\in \mathcal{W}$, a contradiction. Thus, $G'\in\mathcal{A}_{n-2}\setminus \{K_4\}$. Since $n_{2}(G')=n_{2}(G)$ and $n_3(G')=n_3(G)-2$, we have $p(G')=2^{-2}3^{-1/2}p(G)$. Observe that each spanning forest of $G'$ can be extended to a spanning forest of $G$ by at least $F(K_3)=7$ ways obtained by adding any  spanning forest of $G[\{x,y,z\}]\cong K_3$.  By the minimality of $G$, we obtain that $$F(G)\geq 7F(G')\geq7p(G')=7\cdot2^{-2}3^{-1/2}p(G)>p(G),$$ a contradiction, the result follows.
{\hfill}
\end{proof}

We are now ready to prove Theorem \ref{x2}.

\begin{proof}[\em\textbf {Proof of Theorem \ref{x2}}]
Suppose that there exists a graph $G\in \mathcal{A}_{n}\setminus \{K_4\}$  with the minimum number of edges  such that $F(G)<p(G)$. Note that $n\geq4$ and $G$ is $3$-regular by Claim \ref{x-3}.
Now consider an arbitrary edge $uv$ of $G$ and $uv$ is a non-cut edge by Claim \ref{x-2}. Let $u_1,u_2$ and  $v_1,v_2$ be two neighbours of $u$ and $v$ not in $\{u,v\}$, respectively. By Claim \ref{x-5}, we have $u_1u_2, v_1v_2\not\in E(G)$. Let $G'$ be the graph obtained from $G$ by deleting the vertices $u,v$ and adding the edges $u_1u_2$, $v_1v_2$. If $G'=K_4$, then $G=K_{3,3}\in \mathcal{W}$, a contradiction. Therefore, $G'\in\mathcal{A}_{n-2}\setminus \{K_4\}$.

Let $G^{1}$ (resp. $G^{2}$) be  a subgraph induced by the edge set $\{uv,uu_1,uu_2,vv_1,vv_2\}$  (resp. $\{u_1u_2,v_1v_2\}$) in $G$ (resp. $G'$). We consider any spanning forest $T_0$ of $G'-E(G^{2})=G-\{u,v\}$ that can be extended to a spanning forest of $G'$ by adding (possibly $0$) edges in $\{u_1u_2,v_1v_2\}$.   Next we  show that the ratio between $F_{G^{1}}(G,T_0)$ and $F_{G^2}(G',T_0)$  is at least $7$.  The following cases depend on which vertices among    $u_1,u_2,v_1,v_2$  are in a same component.
\begin{itemize}[itemsep=0.5ex, parsep=0.5ex] 
    \item If $u_1,u_2,v_1,v_2$ are in different components of $T_0$, then $F_{G^{2}}(G',T_0)=F(G^2)=4$ and $F_{G^{1}}(G,T_0)=F(G^1)=2^{5}$.
    \item If exactly two of them are in a same component in $T_0$, by symmetry, there are two situations in the following.
First, without loss of generality, assume that $u_1, u_2$ are in a same component in $T_0$ and $v_1,v_2$ are in two other different components, respectively. Then  $F_{G^{2}}(G',T_0)=F(G^2/\{u_1,u_2\})=2$ and $F_{G^{1}}(G,T_0)=F(G^1/\{u_1,u_2\})=24$. Second, without loss of generality, assume that $u_1, v_1$ are in a same component in $T_0$ and $u_2,v_2$ are in two other different components, respectively. Then $F_{G^{2}}(G',T_0)=F(G^2/\{u_1,v_1\})=4$ and $F_{G^{1}}(G,T_0)=F(G^1/\{u_1,v_1\})=4F(K_3)=28$.
\item If  there are two pairs among them  and  each of which contains two vertices within the same component in $T_0$, then there are two cases   by symmetry as follows. First, without loss of generality, assume that $u_1,u_2$ and $v_1,v_2$ are in different components  of $T_0$, respectively. Then we have $F_{G^{2}}(G',T_0)=F\left(\left(G^2/\{u_1,u_2\}\right)/\{v_1,v_2\}\right)=1$ and $F_{G^1}(G,T_0)=F\left(\left(G^1/\{u_1,u_2\}\right)/\{v_1,v_2\}\right)=18$. Second, without loss of generality, assume that $u_1,v_1$ and $u_2,v_2$ are in different components  of $T_0$, respectively.    Then we have $F_{G^{2}}(G',T_0)=F\left(\left(G^2/\{u_1,v_1\}\right)/\{u_2,v_2\}\right)=3$ and $F_{G^1}(G,T_0)=F\left(\left(G^1/\{u_1,v_1\}\right)/\{u_2,v_2\}\right)=F(K_4-e)=24$.
    \item If three of them are in a same component in $T_0$, say $u_1, u_2,v_1$ by symmetry, and $v_2$ is in a different  one, then $F_{G^{2}}(G',T_0)=F(G^2/\{u_1,u_2,v_1\})=2$ and $F_{G^{1}}(G,T_0)=F(G^1/\{u_1,u_2,v_1\})=20$.
    \item If all of $u_1,u_2,v_1,v_2$  are in a same component  of $T_0$, then  we have $F_{G^{2}}(G',T_0)=F(G^2/\{u_1,u_2,v_1,v_2\})=1$ and $F_{G^{1}}(G,T_0)=F(G^1/\{u_1,u_2,v_1,v_2\})=14$.
\end{itemize}
Since $n_{2}(G')=n_{2}(G)$ and $n_3(G')=n_3(G)-2$, we have $p(G')=2^{-2}3^{-1/2}p(G)$. By the minimality of $G$ and Lemma \ref{x876},
$$F(G)\geq7F(G')\geq7p(G')
 =7\cdot2^{-2}3^{-1/2}p(G)
 >p(G),$$
 a contradiction, the result follows.
{\hfill}
\end{proof}

\section{Proof of Theorem \ref{xh-1}}

\begin{figure}[htbp]
\centering
\vspace{-0.8cm}
\hspace{-1cm}
\begin{tikzpicture}[scale=1]
  \coordinate (u) at (-1,1);
  \coordinate (v) at (-1,-1);
  \coordinate (w) at (1,-1);
  \coordinate (z) at (1,1);

  \coordinate (a) at (-0.5,0.5);
  \coordinate (b) at (0.5,0.5);
  \coordinate (c) at (0.5,-0.5);
  \coordinate (d) at (-0.5,-0.5);

  \draw (a) -- (c);
  \draw (a) -- (d);
  \draw (a) -- (b);
  \draw (b) -- (c);
  \draw (b) -- (d);
  \draw (c) -- (d);

  \draw (u) -- (a);
  \draw (u) -- (z);
  \draw (u) -- (v);
  \draw(u)--(w);
  \draw (v) -- (d);
  \draw (v) -- (w);
  \draw (v) -- (z);
  \draw (w) -- (z);
  \draw (w) -- (c);
  \draw(z)--(b);

  \draw (u)..  controls (-0.75,2.5) and (2.5,-0.8) .. (w);

  \draw (v) .. controls (-2.5,-0.8) and (0.75,2.5) .. (z);

  \foreach \point in {u,v,w,z,a,b,c,d} {
    \filldraw [black] (\point) circle (2pt);
  }

  \node[above right] at (-0.5,0.4) {$v_1$};
  \node[below right] at (0.35,0.5) {$v_4$};
  \node[below] at (0.5,-0.5) {$v_3$};
  \node[below ] at (-0.4,-0.5) {$v_2$};
  \node[below] at (0,-1.5) {$H_1$};
 \coordinate (u1) at (3,1);
  \coordinate (v1) at (3,-1);
  \coordinate (w1) at (5,-1);
  \coordinate (z1) at (5,1);

  \coordinate (a1) at (3.5,0.5);
  \coordinate (b1) at (4.5,0.5);
  \coordinate (c1) at (4.5,-0.5);
  \coordinate (d1) at (3.5,-0.5);
\coordinate (e1) at (2.5,0);

  \draw (a1) -- (c1);
  \draw (a1) -- (d1);
  \draw (a1) -- (b1);
  \draw (b1) -- (c1);
  \draw (b1) -- (d1);
  \draw (c1) -- (d1);
\draw(e1)--(u1);
\draw(e1)--(v1);

  \draw (u1) -- (a1);
  \draw (u1) -- (z1);
  \draw (u1) -- (v1);

  \draw (v1) -- (d1);
  \draw (v1) -- (w1);

  \draw (w1) -- (z1);
  \draw (w1) -- (c1);
  \draw(z1)--(b1);

  \draw (e1)..  controls (2,-2) and (4.5,-1.5) .. (w1);

  \draw (e1) .. controls (2,2) and (4.5,1.5) .. (z1);

  \foreach \point in {u1,v1,w1,z1,a1,b1,c1,d1,e1} {
    \filldraw [black] (\point) circle (2pt);
  }

  \node[above ] at (3.5,0.5) {$v_1$};
  \node[above ] at (4.5,0.5) {$v_4$};
  \node[below ] at (4.5,-0.5) {$v_3$};
  \node[below ] at (3.7,-0.5) {$v_2$};
  \node[below] at (4,-1.5) {$H_2$};

  \coordinate (u2) at (7.6,1.3);
  \coordinate (v2) at (7.6,-1.3);
  \coordinate (w2) at (9.5,0);

  \coordinate (a2) at (7.1,0.5);
  \coordinate (b2) at (8.1,0.5);
  \coordinate (c2) at (8.1,-0.5);
  \coordinate (d2) at (7.1,-0.5);

  \draw (a2) -- (c2);
  \draw (a2) -- (d2);
  \draw (a2) -- (b2);
  \draw (b2) -- (c2);
  \draw (b2) -- (d2);
  \draw (c2) -- (d2);

  \draw (u2) -- (a2);
  \draw (u2) -- (b2);
  \draw (u2) -- (w2);

  \draw (v2) -- (d2);
  \draw (v2) -- (c2);
\draw (v2) -- (w2);

  \draw (u2)..  controls (5.7,1) and (5.7,-1) .. (v2);

  \foreach \point in {u2,v2,w2,a2,b2,c2,d2} {
    \filldraw [black] (\point) circle (2pt);
  }

  \node[below left ] at (7.1,0.7) {$v_1$};
  \node[right ] at (8.1,0.4) {$v_4$};
  \node[ right] at (8.1,-0.4) {$v_3$};
  \node[above left] at (7.1,-0.8) {$v_2$};
  \node[below] at (7.6,-1.5) {$H_3$};


  \coordinate (u3) at (10.5,1);
  \coordinate (v3) at (10.5,-1);
  \coordinate (w3) at (12.5,-1);
  \coordinate (z3) at (12.5,1);

  \coordinate (a3) at (11.5,1);
  \coordinate (b3) at (11.5,-1);
  \coordinate (c3) at (11.5,0);

  \draw (u3) -- (a3);
  \draw (u3) -- (c3);
  \draw (u3) -- (b3);
  \draw (u3) -- (v3);

  \draw (v3) -- (a3);
  \draw (v3) -- (b3);
\draw (v3) -- (c3);

 \draw (z3) -- (a3);
\draw (z3) -- (b3);
\draw (z3) -- (c3);
\draw (z3) -- (w3);

 \draw (w3) -- (a3);
\draw (w3) -- (b3);
\draw (w3) -- (c3);
  \foreach \point in {u3,v3,w3,z3,a3,b3,c3} {
    \filldraw [black] (\point) circle (2pt);
  }

  \node[left] at (10.5,1) {$x$};
  \node[left] at (10.5,-1) {$y$};
  \node[above ] at (11.5,1) {$u$};
  \node[above ] at (11.5,0.1) {$v$};
  \node[below ] at (11.5,-1) {$w$};
  \node[below] at (11.5,-1.5) {$H_4$};

  \coordinate (u4) at (-1,-3);
  \coordinate (v4) at (-1,-5);
  \coordinate (w4) at (1,-5);
  \coordinate (z4) at (1,-3);
 \coordinate (y4) at (1,-4);
  \coordinate (a4) at (0,-3);
  \coordinate (b4) at (0,-5);
  \coordinate (c4) at (0,-4);

  \draw (u4) -- (a4);
  \draw (u4) -- (c4);
  \draw (u4) -- (b4);
  \draw (u4) -- (v4);

  \draw (v4) -- (a4);
  \draw (v4) -- (b4);
\draw (v4) -- (c4);

 \draw (z4) -- (a4);
\draw (y4) -- (b4);
\draw (z4) -- (c4);
\draw (z4) -- (w4);

 \draw (y4) -- (a4);
\draw (w4) -- (b4);
\draw (w4) -- (c4);

 \draw (z4) .. controls (1.5,-3.4) and (1.5,-4.5) .. (w4);
  \foreach \point in {u4,v4,w4,z4,a4,b4,c4,y4} {
    \filldraw [black] (\point) circle (2pt);
  }

  \node[left] at (-1,-3) {$x$};
  \node[left] at (-1,-5) {$y$};
  \node[above ] at (0,-3) {$u$};
  \node[above ] at (0,-3.9) {$v$};
  \node[below ] at (0,-5) {$w$};
  \node[below] at (0,-5.5) {$H_5$};

  \coordinate (u5) at (3,-3);
  \coordinate (v5) at (4.5,-5);
  \coordinate (w5) at (3.5,-5);
  \coordinate (z5) at (4,-3);
 \coordinate (y5) at (3.5,-4);
  \coordinate (a5) at (5,-3);
  \coordinate (b5) at (4.5,-4);

  \draw (u5) -- (z5);
  \draw (u5) -- (w5);
  \draw (u5) -- (y5);
  \draw (z5) -- (a5);

  \draw (z5) -- (b5);
  \draw (z5) -- (y5);
\draw (a5) -- (b5);

 \draw (a5) -- (v5);
\draw (w5) -- (v5);
\draw (v5) -- (b5);
\draw (w5) -- (y5);

 \draw (y5) -- (b5);

 \draw (u5) .. controls (4,-4) and (4,-4) .. (v5);
 \draw (a5) .. controls (4,-4) and (4,-4) .. (w5);
  \foreach \point in {u5,v5,w5,z5,a5,b5,y5} {
    \filldraw [black] (\point) circle (2pt);
  }

  \node[left] at (3,-3) {$u$};
  \node[right] at (5,-3) {$v$};
  \node[left ] at (3.5,-5) {$x$};
  \node[right] at (4.5,-5) {$y$};
  \node[below] at (4,-5.5) {$H_6$};

   \coordinate (u6) at (7.1,-3);
  \coordinate (v6) at (6.1,-4);
  \coordinate (w6) at (7.6,-3.65);
  \coordinate (z6) at (8.1,-3);

  \coordinate (a6) at (7.1,-4.15);
  \coordinate (b6) at (8.1,-4.15);
  \coordinate (c6) at (9.1,-4);
  \coordinate (d6) at (8.6,-5);
 \coordinate (e6) at (6.6,-5);

  \draw (u6) -- (z6);
  \draw (u6) -- (v6);
  \draw (u6) -- (w6);
  \draw (z6) -- (w6);
  \draw (v6) -- (a6);
  \draw (v6) -- (e6);

  \draw (e6) -- (a6);
  \draw (w6) -- (a6);
  \draw (w6) -- (b6);
  \draw(a6)--(b6);
  \draw (b6) -- (c6);
  \draw (b6) -- (d6);
  \draw (e6) -- (d6);
  \draw (z6) -- (c6);
  \draw (c6) -- (d6);

  \draw (u6)..  controls (7.6,-3) and (9,-4) .. (d6);

 \draw (v6) .. controls (7.1,-5) and (8.1,-5) .. (c6);

  \draw (e6) .. controls (6.3,-4) and (7.6,-3) .. (z6);

  \foreach \point in {u6,w6,a6,b6,c6,e6} {
    \filldraw [black] (\point) circle (2pt);
  }
\draw[black,fill=white](z6) circle (2pt);
\draw[black,fill=white](v6) circle (2pt);
\draw[black,fill=white](d6) circle (2pt);
  \node[below] at (7.6,-3.65) {$x$};
  \node[below ] at (7.3,-4.15) {$y$};
  \node[below] at (7.9,-4.15){$z$};
  \node[below] at (7.6,-5.5) {$H_7$};

   \coordinate (u7) at (11,-3);
  \coordinate (v7) at (10,-4);
  \coordinate (w7) at (11.5,-3.65);
  \coordinate (z7) at (12,-3);

  \coordinate (a7) at (11,-4.15);
  \coordinate (b7) at (12,-4.15);
  \coordinate (c7) at (13,-4);
  \coordinate (d7) at (12.5,-5);
 \coordinate (e7) at (10.5,-5);

  \draw (u7) -- (z7);
  \draw (u7) -- (v7);
  \draw (u7) -- (w7);
  \draw (z7) -- (w7);
  \draw (v7) -- (a7);
  \draw (v7) -- (e7);

  \draw (d7) -- (a7);
  \draw (w7) -- (a7);
  \draw (w7) -- (b7);
  \draw(a7)--(b7);
  \draw (b7) -- (c7);
  \draw (b7) -- (e7);
  \draw (e7) -- (d7);
  \draw (z7) -- (c7);
  \draw (c7) -- (d7);

  \draw (u7)..  controls (11.5,-3) and (12.9,-4) .. (d7);

 \draw (v7) .. controls (11,-5) and (12,-5) .. (c7);

  \draw (e7) .. controls (10.2,-4) and (11.5,-3) .. (z7);

  \foreach \point in {u7,w7,a7,b7,c7,e7} {
    \filldraw [black] (\point) circle (2pt);
  }
\draw[black,fill=white](z7) circle (2pt);
\draw[black,fill=white](v7) circle (2pt);
\draw[black,fill=white](d7) circle (2pt);
  \node[below] at (11.5,-3.65) {$x$};
  \node[below ] at (11,-4.15) {$y$};
  \node[below] at (12,-4.15){$z$};
  \node[below] at (11.5,-5.5) {$H_8$};
\end{tikzpicture}

\vspace{0.6cm}
\begin{tikzpicture}[scale=1]
 \coordinate (aa3) at (-0.5,-7);
  \coordinate (bb3) at (-0.5,-7.5);
  \coordinate (cc3) at (-0.5,-8);
  \coordinate (dd3) at (-0.5,-8.5);
  \coordinate (ee3) at (-0.5,-9);
  \coordinate (ff3) at (-0.5,-9.5);

  \coordinate (xx3) at (1,-7.5);
  \coordinate (yy3) at (1,-8.25);
  \coordinate (zz3) at (1,-9);

  \draw (aa3) -- (bb3);
  \draw (bb3) -- (cc3);
  \draw (cc3) -- (dd3);
  \draw (dd3) -- (ee3);
  \draw (ee3) -- (ff3);
 \draw (xx3) -- (yy3);
\draw (yy3) -- (zz3);

 \draw (aa3) -- (xx3);
  \draw (bb3) -- (yy3);
  \draw (cc3) -- (zz3);
  \draw (dd3) -- (xx3);
  \draw (ee3) -- (yy3);
 \draw (ff3) -- (zz3);

 \draw (aa3) .. controls (-1,-7.25) and (-1,-7.75) .. (cc3);
\draw (dd3) .. controls (-1,-8.75) and (-1,-9.25) .. (ff3);
\draw (bb3) .. controls (-1,-8) and (-1,-8.5) .. (ee3);
\draw (aa3) .. controls (-1.5,-7.25) and (-1.5,-9.25) .. (ff3);
\draw (xx3) .. controls (1.6,-7.75) and (1.6,-8.75) .. (zz3);

  \foreach \point in {aa3,bb3,cc3,dd3,ee3,ff3,xx3,yy3,zz3} {
    \filldraw [black] (\point) circle (2pt);
  }
  \node[above] at (1,-7.4) {$x$};
  \node[right] at (1,-8.25) {$y$};
  \node[below] at (1,-9.1){$z$};
  \node[below] at (0,-10) {$Z_1$};

\coordinate (aa2) at (3.45,-7);
  \coordinate (bb2) at (3.45,-7.5);
  \coordinate (cc2) at (3.45,-8);
  \coordinate (dd2) at (3.45,-8.5);
  \coordinate (ee2) at (3.45,-9);
  \coordinate (ff2) at (3.45,-9.5);

  \coordinate (xx2) at (4.95,-7.5);
  \coordinate (yy2) at (4.95,-8.25);
  \coordinate (zz2) at (4.95,-9);

  \draw (aa2) -- (bb2);
  \draw (bb2) -- (cc2);
  \draw (cc2) -- (dd2);
  \draw (dd2) -- (ee2);
  \draw (ee2) -- (ff2);
 \draw (xx2) -- (yy2);
\draw (yy2) -- (zz2);

 \draw (aa2) -- (xx2);
  \draw (bb2) -- (yy2);
  \draw (cc2) -- (zz2);
  \draw (dd2) -- (yy2);
  \draw (ee2) -- (xx2);
 \draw (ff2) -- (zz2);

 \draw (aa2) .. controls (2.95,-7.25) and (2.95,-7.75) .. (cc2);
\draw (dd2) .. controls (2.95,-8.75) and (2.95,-9.25) .. (ff2);
\draw (bb2) .. controls (2.95,-8) and (2.95,-8.5) .. (ee2);
\draw (aa2) .. controls (2.45,-7.25) and (2.45,-9.25) .. (ff2);
\draw (xx2) .. controls (5.55,-7.75) and (5.55,-8.75) .. (zz2);

  \foreach \point in {aa2,bb2,cc2,dd2,ee2,ff2,xx2,yy2,zz2} {
    \filldraw [black] (\point) circle (2pt);
  }
  \node[above] at (4.95,-7.4) {$x$};
  \node[right] at (4.95,-8.25) {$y$};
  \node[below] at (4.95,-9.1){$z$};
  \node[below] at (4,-10) {$Z_2$};

\coordinate (aa4) at (7.35,-7);
  \coordinate (bb4) at (7.35,-7.5);
  \coordinate (cc4) at (7.35,-8);
  \coordinate (dd4) at (7.35,-8.5);
  \coordinate (ee4) at (7.35,-9);
  \coordinate (ff4) at (7.35,-9.5);

  \coordinate (xx4) at (8.85,-7.5);
  \coordinate (yy4) at (8.85,-8.25);
  \coordinate (zz4) at (8.85,-9);

  \draw (aa4) -- (bb4);
  \draw (bb4) -- (cc4);
  \draw (cc4) -- (dd4);
  \draw (dd4) -- (ee4);
  \draw (ee4) -- (ff4);
 \draw (xx4) -- (yy4);
\draw (yy4) -- (zz4);

 \draw (aa4) -- (xx4);
  \draw (bb4) -- (yy4);
  \draw (cc4) -- (zz4);
  \draw (dd4) -- (zz4);
  \draw (ee4) -- (yy4);
 \draw (ff4) -- (xx4);

 \draw (aa4) .. controls (6.85,-7.25) and (6.85,-7.75) .. (cc4);
\draw (dd4) .. controls (6.85,-8.75) and (6.85,-9.25) .. (ff4);
\draw (bb4) .. controls (6.85,-8) and (6.85,-8.5) .. (ee4);
\draw (aa4) .. controls (6.35,-7.25) and (6.35,-9.25) .. (ff4);
\draw (xx4) .. controls (9.45,-7.75) and (9.45,-8.75) .. (zz4);

  \foreach \point in {aa4,bb4,cc4,dd4,ee4,ff4,xx4,yy4,zz4} {
    \filldraw [black] (\point) circle (2pt);
  }
  \node[above] at (8.85,-7.4) {$x$};
  \node[right] at (8.85,-8.25) {$y$};
  \node[below] at (8.85,-9.1){$z$};
  \node[below] at (8,-10) {$Z_3$};
  \end{tikzpicture}

\caption{Graphs $H_i$ and $Z_j$ for $i\in[8]$ and $j\in[3]$.}
\end{figure}
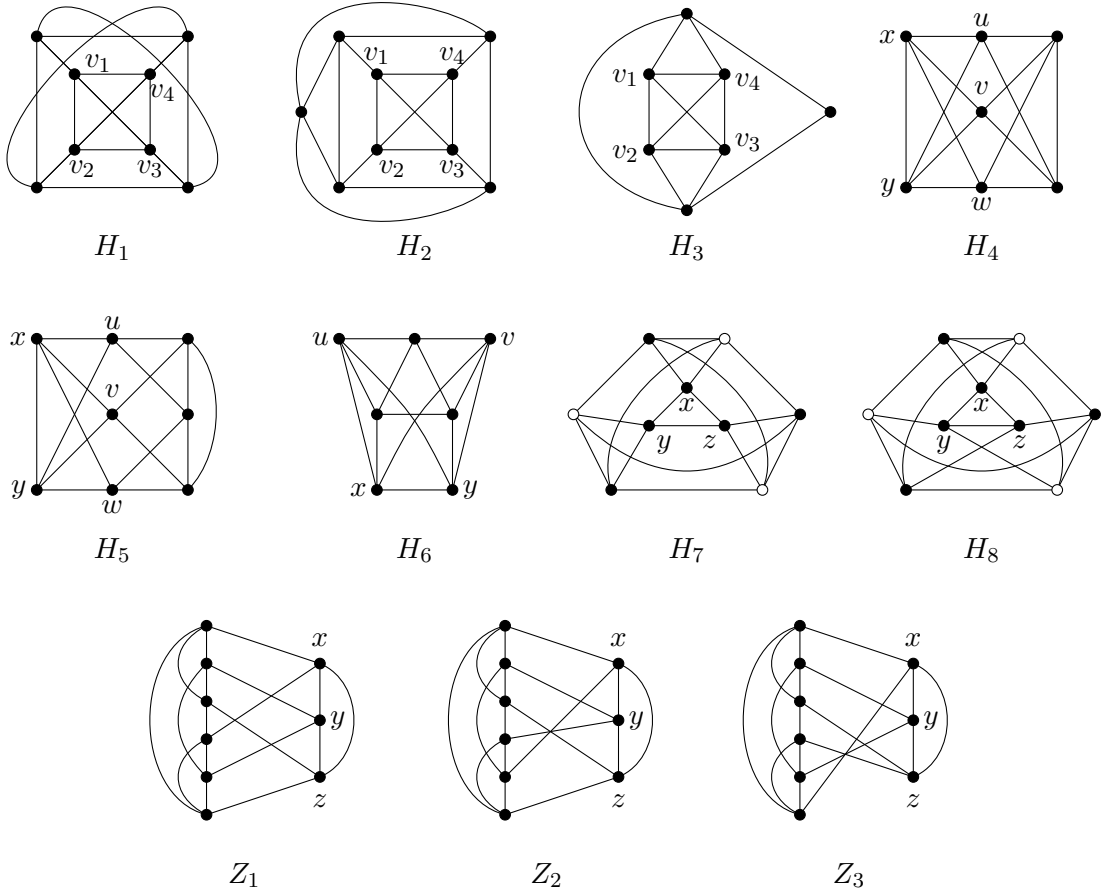

Before proving Theorem \ref{xh-1}, we need to define several graphs.  Let $X_6$ and $X_7$ be two graphs obtained from  $K_5$ and $K_6^{-}$ by subdividing an arbitrary edge (all edges being isomorphic), respectively. The graph $Y_5$ (resp. $Y_5'$) is obtained from $K_4$ by adding a new vertex $w$ and connecting
 $w$ to $2$ (resp. $3$) vertices of $K_4$. The graph $Y_6$ is obtained from $Y_5$ by adding a new vertex $w'$ together with edges joining $w'$ to  $2$ vertices  in $K_4$  that are distinct from that of $w$. Let $Y_6'$ be obtained from $Y_6$ by adding an edge $ww'$. Let $H_1$, $H_2$,...,$H_8$ be eight graphs and  $Z_1,Z_2,Z_3$ be the other three graphs as depicted in  Figure $1$. These graphs will appear in the following proofs. For any graph $G$, we define
 $$q(G)=2^{n_2(G)+\frac{3n_3(G)+n_4(G)-9}{5}}198^{\frac{n_3(G)+2n_4(G)+2}{10}}.$$
  Let us compute $F(G)$ and $q(G)$ for several graphs by using SageMath, see Table \ref{tab:eight_column}.

\begin{table}[htbp]
  \centering
  \begin{tabular}{lccc|cccc}
    \Xhline{1pt}
     $G$     & $F(G)$ &  $q(G)$ & $F(G)\geq q(G)$?  &$G$     & $F(G)$ &  $q(G)$ & $F(G)\geq q(G)$? \\ 
    \Xhline{1pt}
    $K_3$    & $7$      & $2^{\frac{6}{5}}198^{\frac{1}{5}}$ &Yes &  $H_2$ & $52485$ & $198^2$ & Yes  \\
     $K_{4}$   & $38$       &$2^{\frac{3}{5}}198^{\frac{3}{5}}$ &Yes & $H_{3}$ & $2457$ & $2^{\frac{2}{5}}198^{\frac{7}{5}}$ & Yes\\
    $K_5$    & $291$      &$2^{-\frac{4}{5}}198^{\frac{6}{5}}$     &No & $H_{4}$ & $4061$ & $2^{-\frac{2}{5}}198^{\frac{8}{5}}$ & Yes \\
     $K_6^-$   &  $1083$    &$2^{-\frac{3}{5}}198^{\frac{7}{5}}$      &No& $H_{5}$ & $14763$ & $2^{-\frac{1}{5}}198^{\frac{9}{5}}$ & Yes \\
     $X_6$   & $687$      &$2^{\frac{1}{5}}198^{\frac{6}{5}}$ &Yes     & $H_{6}$ & $4019$ & $2^{-\frac{2}{5}}198^{\frac{8}{5}}$ & Yes\\
       $X_7$   & $2527$      &$2^{\frac{2}{5}}198^{\frac{7}{5}}$ &Yes      & $H_7$ &57631 &$198^{2}$ &Yes \\
        $Y_5$   & $128$      &$2^{\frac{4}{5}}198^{\frac{4}{5}}$ &Yes & $H_8$ & 58975 & $198^{2}$ &Yes \\
         $Y_5'$   & $198$      &$198$ &Yes &   $Z_1$ & $57631$  & $198^{2}$ & Yes \\
          $Y_6$   & $431$      &$396$ &Yes & $Z_2$ &$58417$  & $198^{2}$ & Yes \\
          $Y_6'$    & $722$      &$2^{\frac{1}{5}}198^{\frac{6}{5}}$     &Yes  &$Z_3$ & $56101$ &$198^{2}$ & Yes \\
          $H_1$  & $14381$ & $2^{-\frac{1}{5}}198^{\frac{9}{5}}$ &Yes\\
   \Xhline{1pt}
  \end{tabular}
  \caption{The values of $F(G)$ and $q(G)$ for several graphs.}
  \label{tab:eight_column}
\end{table}

\begin{Remark}\label{x634}
Theorem \ref{xh-1} is true for all graphs  except $K_5$ and $K_6^-$ in Table \ref{tab:eight_column}.
 \end{Remark}

Let $\mathcal{M}$ be the set of all graphs  except $K_5$ and $K_6^-$ in Table \ref{tab:eight_column}. By Remark \ref{x634}, Theorem \ref{xh-1} is true for any graph $G\in \mathcal{M}$.  For each $n\geq 3$, let $\mathcal{B}_{n}$ be the set of all connected graphs of order $n$, whose vertex degrees are from $\{2,3,4\}$. Suppose that  $G\in \mathcal{B}_{n}\setminus \mathcal{E}$ is a graph with the minimum number of edges  such that $F(G)<q(G)$. Then $G\not\in \mathcal{M}$ and $n\geq4$ since $G\neq K_3$. Now  we show  some claims as follows.

\begin{claim}\label{xh-2}
$G$ has no cut-edge.
\end{claim}
\begin{proof}[\em\textbf {Proof of Claim \ref{xh-2}}]
 Suppose on the contrary that $e$ is a cut-edge of $G$. There exists a path $P=v_{1}v_{2}\cdots v_s$ of length $s-1$ in $G$ that contains $e$, where $d(v_1),d(v_s)\geq3$ and $d(v_2)=d(v_3)=\cdots=d(v_{s-1})=2$. Note that  all the edges of $P$ are cut-edges. Deleting the interior vertices of $P$ forms two connected components $G_1$ and $G_2$. Both $G_1$ and $G_2$  have a vertex of degree at least $2$ and at most 3 (namely, $v_1$ and $v_s$), and   hence each of them belongs to some $\mathcal{B}_{k}\setminus \{K_5,K_6^{-}\}$ where $k\geq 3$. Whatever the degrees of $v_1$ and $v_s$, we always have $q(G_1)q(G_2)=2^{-(s-1)}q(G)$.   By the minimality of $G$, we have
$$F(G)=2^{s-1}F(G_1)F(G_2)\geq 2^{s-1}q(G_1)q(G_2)
=q(G).$$
This contradiction concludes Claim \ref{xh-2}.
{\hfill}
\end{proof}

\begin{claim}\label{xh-3}
$G$ has no cut-vertex.
\end{claim}
\begin{proof}[\em\textbf {Proof of Claim \ref{xh-3}}]
Suppose that $u$ is a cut-vertex of $G$. Since $G$ has no cut-edge by Claim \ref{xh-2}, we necessarily have $d_{G}(u)=4$, and $G-u$ has exactly $2$ components $Q_{1}$ and $Q_{2}$, each of which contains precisely $2$ neighbours of $u$. For $i=1,2$, set $G_{i}=Q_{i}+u$, so $d_{G_i}(u)=2$. Then each of $G_{1}$ and $G_{2}$  belongs to some $\mathcal{B}_{k}\setminus \{K_5,K_6^{-}\}$ for $k\geq3$. Since $n_{2}(G_1)+n_2(G_2)=n_{2}(G)+2$, $n_3(G_1)+n_3(G_2)=n_3(G)$ and $n_4(G_1)+n_4(G_2)=n_4(G)-1$, we have $q(G_1)q(G_2)=2^{2-\frac{10}{5}}\cdot198^{\frac{-2}{10}+\frac{2}{10}}q(G)=q(G)$. By the minimality of $G$, we have
$$F(G)=F(G_1)F(G_2)\geq q(G_1)q(G_2)=q(G),$$
a contradiction, the result follows.
{\hfill}
\end{proof}

\begin{claim}\label{xh-4}
If $v\in V(G)$ has $d_G(v)=2$, then $v$ is contained in a triangle of $G$.
\end{claim}
\begin{proof}[\em\textbf {Proof of Claim \ref{xh-4}}]
Suppose that $v\in V(G)$ has $d_G(v)=2$ with neighbours $x$ and $y$ such that $xy\not\in E(G)$. Let $G'=(G-v)\cup \{xy\}$, obviously $G'\in \mathcal{B}_{n-1}$. Since $n_{2}(G')=n_{2}(G)-1$, $n_3(G')=n_3(G)$ and $n_4(G')=n_4(G)$, we have $q(G')=2^{-1}q(G)$. If $G'\not\in\mathcal{E}$,  by the minimality of $G$ and Lemma \ref{x322}, we have
$$F(G)\geq2F(G')\geq 2q(G')=q(G),$$
a contradiction. If $G'\in \mathcal{E}$, then $G\in \{X_6,X_7\}\subseteq \mathcal{M}$, a contradiction.  This completes the proof of Claim \ref{xh-4}.
{\hfill}
\end{proof}

We continue with proving some properties which will be useful for us later on.

\begin{claim}\label{xh-5}
 $G$ has no subgraph isomorphic to $K_4$.
\end{claim}
\begin{proof}[\em\textbf {Proof of Claim \ref{xh-5}}]
Suppose that $V'=\{v_1,v_2,v_3,v_4\}$ induces a $K_4$ in $G$. For $i\in [4]$, let $u_i$ be the neighbour of $v_i$ not in $V'$, if it exists. We  assume that $u_{1}$ exists because otherwise $G$ is $K_{4}\in \mathcal{M}$, a contradiction. By Claim \ref{xh-3} and $G\not\cong K_{5}$, we have either $G$ is obtained  by adding $r\in\{2,3\}$   edges between $u_{1}$ and $K_{4}$, or  $u_{2}$ exists and $u_{2}\neq u_{1}$. The former has two cases that $G\in \{Y_5,Y_5'\}\subseteq \mathcal{M}$, a contradiction.

It thus remains to deal with the latter. We define  $A$  as the maximum number of neighbours in $V'$ of a vertex in $V(G)- V'$. Since $G\not\in \mathcal{E}$, we have $A\in\{1,2,3\}$. Then we consider the following three cases.

\medskip
{\indent\bf Case 1.} $A=1$. We construct $G'$ by contracting $V'$ into a new vertex $v'$. Because $A=1$ and $u_{1}$, $u_{2}$ exist with $u_{1}\neq u_{2}$, we have $G'\in \mathcal{B}_{n-3}$. Regardless of the existence of $u_3$ and  $u_4$, we always have $q(G')=2^{-\frac{3}{5}}198^{-\frac{3}{5}}q(G)$.
Note that $G'\not\in \mathcal{E}$, for otherwise $G\in \{H_1,H_2\}\subseteq \mathcal{M}$, a contradiction.  Hence, it follows that $$F(G)\geq F(K_4)F(G')=38F(G')\geq38 \cdot2^{-\frac{3}{5}}198^{-\frac{3}{5}}q(G)> q(G),$$ a contradiction.

\medskip
{\indent\bf Case 2.} $A=3$. Without loss of generality, we assume that $u_{2}=u_{3}=u_{4}$ and call this vertex  $w$. Since $d_G(u_{1})\geq2$ and $u_1$ is not a cut-vertex by Claim \ref{xh-3}, we have $d_{G}(w)=4$. Moreover, $w$ and $u_{1}$ are not adjacent, for otherwise $G=X_6\in \mathcal{M}$, a contradiction.  Now we form $G'$ by contracting $V'\cup \{w\}$ into a new vertex $v'$ of degree $2$. So we infer that $G'\in \mathcal{B}_{n-4}\setminus \mathcal{E}$. Observe that  $F(G[V'\cup \{w\}])=F(Y_5')=198$ and $q(G')=198^{-1}q(G)$.  It follows that $$F(G)\geq F(G[V'\cup \{w\}])F(G')\geq 198\cdot 198^{-1}q(G)=q(G),$$ a contradiction.

\medskip
{\indent\bf Case 3.} $A=2$. Without loss of generality, we   assume  that $u_{2}=u_{3}$ and  call this vertex $w$. Next we consider the following three cases.

{\indent \bf Subcase 3.1.} If  $u_{4}$ does not  exist, then $d_{G}(w)\geq 3$, for otherwise $u_{1}$ is a cut-vertex,  contradicting Claim \ref{xh-3}. We form $G'$ by contracting $V'$ into a new vertex $v'$ and deleting multiple edges that occur. So  $G'\in \mathcal{B}_{n-3}\setminus \mathcal{E}$ as $d_{G'}(v')=2$, and $F(G')\geq q(G')=2^{\frac{1}{5}}198^{-\frac{4}{5}}q(G)$ (whatever the degree of $w$).  Let $G^{1}$ be a subgraph induced by the edge set $E(G[V'\cup \{w\}])\cup \{v_1u_1\}$  in $G$, and let $G^{2}$ be a subgraph induced by the edge set $\{v'u_1,v'w\}$  in $G'$.
We consider any spanning forest $T_0$ of $G'-v'=G-V'$ that can be extended to a spanning forest of $G'$ by adding (possibly $0$) edges in $\{v'u_1,v'w\}$.  We next show that the ratio between $F_{G^{1}}(G,T_0)$ and $F_{G^2}(G',T_0)$  is at least  $64$.
\begin{itemize}[itemsep=0.5ex, parsep=0.5ex] 
    \item If $u_1,w$ are in different components of $T_0$, then $F_{G^{2}}(G',T_0)=F(G^{2})=4$ and $F_{G^{1}}(G,T_0)=F(G^{1})=2F(Y_5)=256$.
    \item If $u_1,w$ are   in a same component  of $T_0$, then $F_{G^{2}}(G',T_0)=F(G^{2}/\{u_1,w\})=3$ and $F_{G^{1}}(G,T_0)=F(G^{1}/\{u_1,w\})=F(Y_5')=198$.
\end{itemize}
Thus, we conclude that $$F(G)\geq 64F(G')\geq64\cdot2^{\frac{1}{5}}198^{-\frac{4}{5}}q(G)>q(G),$$ a contradiction.

{\indent \bf Subcase 3.2.} If  $u_{4}$  exists and $u_1\neq u_4$, we first assume that $d_{G}(w)=2$.   We form $G'$ by contracting  $V'\cup\{w\}$ into a new vertex $v'$, observe that $G'\in \mathcal{B}_{n-4}\setminus \mathcal{E}$ as $d_{G'}(v')=2$. Note that $F(G')\geq q(G')=2^{-\frac{4}{5}}198^{-\frac{4}{5}}q(G)$ and  $F(G[V'\cup \{w\}])=F(Y_5)=128$. Hence, $F(G)\geq F(G[V'\cup \{w\}])F(G')\geq128\cdot2^{-\frac{4}{5}}198^{-\frac{4}{5}}q(G)>q(G)$, a contradiction. Now suppose  that $d_{G}(w)\geq3$ and  let $G'$ be  obtained from $G$ by contracting  $V'$ into a new vertex $v'$ and deleting multiple edges that occur. Note that $G'\in \mathcal{B}_{n-3}\setminus \mathcal{E}$ as $d_{G'}(v')=3$ and $F(G')\geq q(G')=2^{\frac{1}{5}}198^{-\frac{4}{5}}q(G)$ (whatever the degree of $w$). Let $G^1$ be a subgraph induced by the edge set $E(G[V'\cup \{w\}])\cup \{v_1u_1,v_4u_4\}$  in $G$,  and let $G^2$ be a subgraph induced by the edge set $\{v'u_1,v'u_4,v'w\}$  in $G'$. We consider any spanning forest $T_0$ of $G'-v'=G-V'$ that can be extended to a spanning forest of $G'$ by adding (possibly $0$) edges in $\{v'u_1,v'u_4,v'w\}$.   We next show that the ratio between $F_{G^{1}}(G,T_0)$ and $F_{G^2}(G',T_0)$  is at least $64$.
\begin{itemize}[itemsep=0.5ex, parsep=0.5ex] 
    \item Assume that $u_1,u_4$ and $w$ are in different components of $T_0$. Then $F_{G^{2}}(G',T_0)=F(G^{2})=8$ and $F_{G^{1}}(G,T_0)=F(G^{1})=4F(Y_5)=512$.
    \item Assume that $u_1$ and $u_4$ are   in a same component  of $T_0$ and $w$ is in a different one. Then $F_{G^{2}}(G',T_0)=F(G^{2}/\{u_1,u_4\})=6$ and $F_{G^{1}}(G,T_0)=F(G^{1}/\{u_1,u_4\})=F(Y_6)=431$.
        \item If $w$ and $u$ are   in a same component  of $T_0$ and $u'$ is in a different one, where $\{u,u'\}=\{u_1,u_4\}$, then assume without loss of generality that $u=u_1$ and $u'=u_4$. Then $F_{G^{2}}(G',T_0)=F(G^{2}/\{u_1,w\})=6$ and $F_{G^{1}}(G,T_0)=F(G^{1}/\{u_1,w\})=2F(Y_5')=396$.
        \item Assume that $u_1,u_4$ and $w$ are   in a same component  of $T_0$. Then $F_{G^{2}}(G',T_0)=F(G^{2}/\{u_1,u_4,w\})=4$ and $F_{G^{1}}(G,T_0)=F(G^{1}/\{u_1,u_4,w\})=F(K_5)=291$.
\end{itemize}
 Consequently, we infer  that $$F(G)\geq 64 F(G')\geq64\cdot2^{\frac{1}{5}}198^{-\frac{4}{5}}q(G)>q(G),$$ a contradiction.

{\indent \bf Subcase 3.3.} If  $u_{4}$  exists and $u_1= u_4$, then  we call this vertex $w'$. Let $H$ be the subgraph of $G$ induced by $V'\cup \{w,w'\}$. Observe that $H$ cannot be the whole graph $G$, for otherwise $G\in \{Y_6,Y_6'\}\subseteq \mathcal{M}$, a contradiction.  By  Claim \ref{xh-3}, we have $d_G(w),d_G(w')\geq3$.
If $ww'\in E(G)$, then Claim \ref{xh-3} implies that $d_G(w)=d_G(w')=4$. If $w$ and $w'$ have a common neighbour, then $G=H_3\in \mathcal{M}$, a contradiction. Now contracting $V'\cup \{w,w'\}$ into a new vertex $v'$ yields a graph $G'\in \mathcal{B}_{n-5}\setminus \mathcal{E}$ as $d_{G'}(v')=2$. Note that $F(G')\geq q(G')=2^{-\frac{1}{5}}198^{-\frac{6}{5}}q(G)$ and $F(G[V'\cup \{w,w'\}])=F(Y_6')=722$. Hence, we infer that
$$F(G)\geq F(G[V'\cup \{w,w'\}])F(G')\geq722\cdot2^{-\frac{1}{5}}198^{-\frac{6}{5}}q(G)>q(G),$$
a contradiction.

Now suppose that $ww'\not\in E(G)$. Let $G'=(G-V')\cup \{ww'\}$ and observe that $G'\in \mathcal{B}_{n-4}\setminus \mathcal{E}$ as $d_{G'}(w)\leq 3$. Note that $F(G')\geq q(G')=198^{-1}q(G)$ (whatever the degrees of $w$ and $w'$). In this case let $G^3=G[V'\cup \{w,w'\}]$.
 We consider any spanning forest $T_0$ of $G'-ww'=G-V'$ that can be extended to a spanning forest of $G'$ by adding (possibly $0$) edges in $\{ww'\}$.   We next show that the ratio between $F_{G^3}(G,T_0)$ and $F_{ww'}(G',T_0)$
is at least $\frac{431}{2}$.
\begin{itemize}[itemsep=0.5ex, parsep=0.5ex] 
    \item If $w$ and $w'$ are in different components of $T_0$, then $F_{ww'}(G',T_0)=F(ww')=2$ and $F_{G^3}(G,T_0)=F(G^3)=F(Y_6)=431$.
    \item Assume that $w$ and $w'$ are   in a same component  of $T_0$. Then there is exactly $1$ spanning forest of $G'$ that contains $T_0$,  that is, $T_0$ itself.  On the other hand, $F_{G^3}(G,T_0)=F(G^3/\{w,w'\})=F(K_5)=291$.
\end{itemize}
 Consequently, we infer  that
 $$F(G)\geq \frac{431}{2}F(G')\geq\frac{431}{2}\cdot 198^{-1}q(G)>q(G),$$
  a contradiction.
This completes the proof of Claim \ref{xh-5}.
{\hfill}
\end{proof}

Our next aim is to prove that $G$ has no subgraph isomorphic to the diamond $D_4$. Let $D_5$ be the graph obtained from the diamond $D_4$ with the edge set $\{ux,uy,vx,vy,xy\}$ by adding a new vertex $w$ adjacent only to $x$ and to $y$.  We proceed in two steps.

\begin{claim}\label{xh-6}
 $G$ has no subgraph isomorphic to $D_5$.
\end{claim}
\begin{proof}[\em\textbf {Proof of Claim \ref{xh-6}}]
 Suppose that $G$ contains a subgraph  isomorphic to $D_5$ with the edge set $\{ux,uy,vx,vy,wx,wy,xy\}$.  By Claim \ref{xh-5}, $\{u,v,w\}$ is an independent set of $G$.
We form $G'$ from $G$ by deleting the vertices $x$ and $y$, and then adding the edges $uv$, $uw$ and $vw$. Observe that $G'\in \mathcal{B}_{n-2}$ and $q(G')=2^{-\frac{2}{5}}198^{-\frac{2}{5}}q(G)$. If $G'\in \mathcal{E}$, then $G\in \{H_4,H_5\}\subseteq \mathcal{M}$, a contradiction. We assume that $G'\not\in  \mathcal{E}$ and consider any spanning forest $T_0$ of $G'-\{uv,uw,vw\}=G-\{x,y\}$ that can be extended to a spanning forest of $G'$ by adding (possibly $0$) edges in $\{uv,uw,vw\}$.  Let $D^1=G'[\{u,v,w\}]$, which is isomorphic to $K_3$. We next show that the ratio between $F_{D_5}(G,T_0)$ and $F_{D^1}(G',T_0)$ is at least $\frac{81}{7}$.
\begin{itemize}[itemsep=0.5ex, parsep=0.5ex] 
    \item Assume that $u,v$ and $w$ are in different components of $T_0$.  Then $F_{D^1}(G',T_0)=F(D^1)=F(K_3)=7$ and $F_{D_5}(G,T_0)=F(D_5)=81$.
    \item By symmetry, without loss of generality, we assume that  $u,v$ are in a same component of $T_0$  and $w$ is in a different one. Then $F_{D^1}(G',T_0)=F(D^1/\{u,v\})=3$ and $F_{D_5}(G,T_0)=F(D_5/\{u,v\})=47$.
        \item Assume that all of $u,v$ and $w$ are   in a same component  of $T_0$.  Then $F_{D^1}(G',T_0)=F(D^1/\{u,v,w\})=1$
         and $F_{D_5}(G,T_0)=F(D_5/\{u,v,w\})=23$.
\end{itemize}
 Consequently, we infer  that $$F(G)\geq \frac{81}{7}F(G')\geq\frac{81}{7}\cdot2^{-\frac{2}{5}}198^{-\frac{2}{5}}q(G)>q(G),$$ a contradiction. This completes the proof of Claim \ref{xh-6}.
{\hfill}
\end{proof}

\begin{claim}\label{xh-7}
 $G$ has no subgraph isomorphic to $D_4$.
\end{claim}
\begin{proof}[\em\textbf {Proof of Claim \ref{xh-7}}]
 Suppose that $G$ contains a subgraph isomorphic to $D_4$ with the edge set $\{ux,uy,vx,vy,xy\}$.  By Claim \ref{xh-6}, we obtain $N_G(x)\cap N_G(y)=\{u,v\}$. Let $G'$ be obtained from $G\cup\{uv\}$ by  contracting the edge $xy$ into a new vertex $z$ and deleting multiple edges that occur,  so $\{u,v,z\}$ induces a triangle in $G'$. Observe that $G'\in \mathcal{B}_{n-1}$ and $q(G')=2^{-\frac{1}{5}}198^{-\frac{1}{5}}q(G)$. If $G'\in \mathcal{E}$, then $G=K_6^{-}\in\mathcal{ E}$ or $G=H_{6}\in \mathcal{M}$, a contradiction.  Now assume that $G'\not\in  \mathcal{E}$. We consider any spanning forest $T_0$ of $G'-\{uv,uz,vz\}$  that can be extended to a spanning forest of $G'$ by adding (possibly $0$) edges in $\{uv,uz,vz\}$. Note that $T_0$ corresponds to a spanning forest $T_0'$ in $G$ which separates $x$ and $y$.  Let $D^1=G'[\{u,v,z\}]$. We next show that the ratio between $F_{D_4}(G,T_0')$ and $F_{D^1}(G',T_0)$ is at least $\frac{10}{3}$.
\begin{itemize}[itemsep=0.5ex, parsep=0.5ex] 
   \item If $u,v$ and $z$ are in different components of $T_0$, then $u,v,x$ and $y$ are  in different components of $T_0'$. Then $F_{D^1}(G',T_0)=F(D^1)=7$ and $F_{D_4}(G,T_0')=F(D_4)=24$.
    \item If $u$ and $v$ are   in a same component  of $T_0$ and $z$ is in a different one, then $u$ and $v$ are in a same component  of $T_0'$ and each of  $x$ and $y$ is in  a different one, respectively.
          Hence, $F_{D^1}(G',T_0)=F(D^1/\{u,v\})=3$ and $F_{D_4}(G,T_0')=F(D_4/\{u,v\})=14$.
        \item If $z$ is in the same component of $T_0$ as exactly one of $u$ and $v$, say $u$ by symmetry, then, without loss of generality,  assume that $u$ and $x$ are in a same component  of $T_0'$ and  $v,y$ are in two other different components.  Hence, $F_{D^1}(G',T_0)=F(D^1/\{u,z\})=3$ and $F_{D_4}(G,T_0')=F(D_4/\{u,x\})=10$.
        \item If all of $u,v$ and $z$ are   in a same component  of $T_0$, then, without loss of generality,  assume that $u,v$ and $x$ are in a same component  of $T_0'$ and $y$ is in a different one. Then $F_{D^1}(G',T_0)=F(D^1/\{u,v,z\})=1$ and $F_{D_4}(G,T_0')=F(D_4/\{u,v,x\})=4$.
\end{itemize}
 Consequently, we infer  that $$F(G)\geq \frac{10}{3}F(G')\geq\frac{10}{3}\cdot2^{-\frac{1}{5}}198^{-\frac{1}{5}}q(G)>q(G),$$ a contradiction. This completes the proof of Claim \ref{xh-7}.
{\hfill}
\end{proof}

\begin{claim}\label{xh-8}
 $G$ does not contain a triangle.
\end{claim}
\begin{proof}[\em\textbf {Proof of Claim \ref{xh-8}}]
  Suppose that $V'=\{x,y,z\}$ induces a triangle in $G$.  Claim \ref{xh-7} implies that every vertex of $G$ not in $V'$ has at most $1$ neighbour in $V'$. The following three cases depend on  the degree sum of $x$, $y$ and $z$.

\medskip
{\indent\bf Case 1.} $d_G(x)+d_G(y)+d_G(z)\leq10$. By Claim \ref{xh-3} and $n\geq4$, we have $2\leq|\left(N_G(x)\cup N_G(y)   \cup N_G(z)\right)\setminus V'|\leq 4$. Now we form $G'$ from $G$ by contracting $V'$ into a new vertex $v'$ and deleting multiple edges that occur, then $G'\in \mathcal{B}_{n-2}$. Regardless of the  degrees of vertices in $V'$, we always have $q(G')=2^{-\frac{6}{5}}198^{-\frac{1}{5}}q(G)$. If $G'\not\in \mathcal{E}$, then
$$F(G)\geq F(K_3)F(G')\geq7q(G')=7\cdot2^{-\frac{6}{5}}198^{-\frac{1}{5}}q(G)>q(G),$$
 a contradiction. If $G'\in\mathcal{ E}$, then this contradicts  Claims \ref{xh-5} and \ref{xh-7}. Indeed, if $G'=K_5$, let $a,b,c,d$ be the $4$ vertices of $G$ outside $V'$, necessarily $G[\{a,b,c,d\}]\cong K_4$; if $G'=K_6^{-}$, then $G'-v'=K_6^{-}-v'$ is a subgraph of $G$ which contains a copy of $D_4$.

\medskip
{\indent\bf Case 2.} $d_G(x)+d_G(y)+d_G(z)=11$.
Without loss of generality, assume that  $d_G(x)=d_G(y)=4$ and  $d_G(z)=3$.
Let $G_0=G-z$ and $w\in N_G(z)\setminus V'$.
By Claim $\ref{xh-7}$, the vertex $w$ is adjacent to neither $x$ nor $y$, since otherwise
$G$ would contain a subgraph isomorphic to $D_4$. If $d_G(w)=2$, then $w$ is contained in a triangle by
Claim $\ref{xh-4}$. Since $N_G(z)=\{x,y,w\}$, this
would force $w$ to be adjacent to $x$ or $y$, a contradiction. Hence, $d_G(w)\ge 3$.

Let $G_0'=G_0-xy$ and $ G_0''=G_0/xy$. We first show that $G_0'$ is connected. Since $z$ is not a cut-vertex of $G$,
the graph $G_0=G-z$ is connected. If $G_0'$ were disconnected, then $xy$
would be a cut-edge of $G_0$. Let $L_x$ and $L_y$ be the two components of
$G_0'$ with $x\in L_x$ and $y\in L_y$. Since $d_G(x)=d_G(y)=4$, both
$L_x-\{x\}$ and $L_y-\{y\}$ are nonempty. If $w\in L_x$, then $y$ is a cut-vertex
of $G$, and if $w\in L_y$, then $x$ is a cut-vertex of $G$. In either case, we get a
contradiction to Claim $\ref{xh-3}$. Thus $G_0'$ is connected. Then $G'_0\in \mathcal{B}_{n-1}\setminus \mathcal E$ and
$G''_0\in \mathcal{B}_{n-2}\setminus \mathcal E$ since $2\leq d_{G_0'}(w)=d_{G_0''}(w)=d_{G_0}(w)\leq3$. Moreover,
$q(G'_0)=2^{\frac{7}{5}}198^{-\frac{3}{5}}q(G)$ and $q(G''_0)=2^{-\frac{2}{5}}198^{-\frac{2}{5}}q(G)$.

Let $G_{xz}$ be the graph obtained from $G-\{xy,zy\}$ by contracting
the edge $xz$, and let $G_{zy}$ be the graph obtained from
$G-\{xy,xz\}$ by contracting the edge $zy$. Observe that
$G_{xz},G_{zy}\in \mathcal{B}_{n-1}\setminus \mathcal E$ since $G$ contains no $D_4$ and
$q(G_{xz})=q(G_{zy})=2^{\frac{3}{5}}198^{-\frac{2}{5}}q(G)$.

Now we count spanning forests of $G$ according to the number of edges
which belong to $\{xy,xz,zy\}$. First, consider spanning forests containing no edge from $\{xy,xz,zy\}$.
Each spanning forest of $G'_0$ gives two such spanning forests of $G$, according
to whether the edge $zw$ is added or not. Hence the number of such spanning
forests is at least
$2F(G'_0)$. Second, consider spanning forests containing exactly one edge from
$\{xy,xz,zy\}$. If this edge is $xy$, then we start with a spanning forest
of $G'_0$ in which $x$ and $y$ lie in different components. By Lemma
\ref{x11}, the number of such spanning forests is $F(G''_0)$. For each
of them, after reinserting $z$ and adding the edge $xy$, the vertex $z$
can again either be left isolated or joined to $w$ by the edge $zw$.
This gives at least
$2F(G''_0)$
spanning forests of $G$.  If the
unique edge is $xz$ or $zy$, then contracting this prescribed edge gives
at least $F(G_{xz})$ or $F(G_{zy})$ spanning forests, respectively.
Therefore, the number of spanning forests containing exactly one edge from
$\{xy,xz,zy\}$ is at least
$
2F(G''_0)+F(G_{xz})+F(G_{zy})$. Third, consider spanning forests containing exactly two edges from
$\{xy,xz,zy\}$. Start with a spanning forest of $G'_0$ in which $x$ and
$y$ lie in different components. By Lemma \ref{x11}, there are exactly
$F(G''_0)$ such spanning forests. For each of them, reinsert the vertex
$z$, add no edge $zw$, and add an one of the three pairs
$
\{xy,xz\},\{xy,zy\},\{xz,zy\}
$.
Since $x$ and $y$ lie in different components before these two triangle
edges are added, no cycle is created. Hence we obtain at least
$3F(G''_0)$
spanning forests of this type.

The above three families are pairwise disjoint. Consequently,
\[
\begin{aligned}
F(G)
&\ge 2F(G'_0)+5F(G''_0)+F(G_{xz})+F(G_{zy}) \\
&\ge
\left(
2\cdot 2^{\frac{7}{5}}198^{-\frac{3}{5}}
+5\cdot 2^{-\frac{2}{5}}198^{-\frac{2}{5}}
+2\cdot 2^{\frac{3}{5}}198^{-\frac{2}{5}}
\right)q(G) \\
&> q(G),
\end{aligned}
\]
a contradiction.

\medskip
{\indent\bf Case 3.} $d_G(x)+d_G(y)+d_G(z)=12$. Then $d_G(x)=d_G(y)=d_G(z)=4$. Now we consider the following three types of spanning forests of $G$.

{\indent \bf  Type \uppercase\expandafter{\romannumeral 1}.} 
The spanning forest  of $G$ contains no edge in $\{xy,zx,yz\}$. Let $G_{xyz}=G-\{xy,zx,yz\}$. We first show that $G_{xyz}$ is connected. Indeed, otherwise some component of
 $G_{xyz}$ contains exactly one vertex $v$ of $V'$. Since
$d_G(v)=4$, this component contains a vertex outside $V'$, and hence deleting
$v$ disconnects $G$, contradicting Claim \ref{xh-3}. Then $G_{xyz}\in \mathcal{B}_n\setminus \mathcal{E}$ since $d_{G_{xyz}}(x)=2$
and  $q(G_{xyz})=2^{12/5}198^{-3/5}q(G)$. By the minimality of $G$, we have
$
F(G_{xyz})\ge q(G_{xyz})
=
2^{12/5}198^{-3/5}q(G)$. On the other hand, every spanning forest of $G_{xyz}$ is exactly a spanning forest of $G$ containing
no edge in $\{xy,zx,yz\}$. Then the number of such spanning forests is at least
$
2^{12/5}198^{-3/5}q(G)$.

{\indent \bf  Type \uppercase\expandafter{\romannumeral 2}.} 
The spanning forest  of $G$ contains exactly one edge in $\{xy,zx,yz\}$. Let $G_{xy}$ be  obtained from $G-\{zx,zy\}$ by contracting the edge $xy$ into a new vertex. Then $G_{xy}\in\mathcal{B}_{n-1}\setminus \mathcal{E}$ since $G$ has no cut-vertex and $d_{G_{xy}}(z)=2$.  Using an analogous definition where we contract $zx$ or $yz$ instead of $xy$, we  get $G_{zx}$ and $G_{yz}$ which both belong to $\mathcal{B}_{n-1}\setminus \mathcal{E}$. Note that $q(G^{*})=2^{\frac{3}{5}}198^{-\frac{2}{5}}q(G)$ for $G^{*}\in\{G_{xy},G_{zx},G_{yz}\}$. The total number of such spanning forests of $G$ is at least $F(G_{xy})+F(G_{zx})+F(G_{yz})\geq 3\cdot2^{\frac{3}{5}}198^{-\frac{2}{5}}q(G)$.

{\indent \bf  Type \uppercase\expandafter{\romannumeral 3}.} 
The spanning forest  of $G$ contains exactly two edges in $\{xy,zx,yz\}$. Let $G'$ be obtained from $G$ by contracting $V'$ into a new vertex $u$ with degree $6$. Let $N=\left(N_G(x)\cup N_G(y) \cup N_G(z)\right)\setminus\{x,y,z\}$. We first claim that $G'$ can be performed a complete lift of $u$, unless $G$ has precisely $9$ vertices and $G[N]\cong K_{3,3}$. Indeed,  $G[N]$ has maximum degree  at most $3$, and if $|V(G)|>9$, then  $G[N]$ has at least $2$ vertices with degree at most $2$  because $G$ has no cut-vertex.   Since $G$ contains no  $D_4$ by Claim \ref{xh-7},  the complement of $G[N]$ has a perfect matching unless $G$ has $9$ vertices and $G[N]\cong K_{3,3}$.   Furthermore, we obtain that $G\in \{H_7,H_8\}\subseteq \mathcal{M}$ if $G[N]\cong K_{3,3}$, a contradiction.  Then $u$ can be completely lifted in $G'$, which amounts to deleting $u$ and adding  a perfect matching between its neighbours.  Since none of  $x$, $y$ and $z$ is a cut-vertex in $G$, the graph $G'-u$ has at most $3$ components. By adding a suitable  perfect matching in the complement of  $G[N]$,  there must exist a complete lift of $u$ in $G'$ yielding a connected graph $G'_u$.   Note that $G'_u\in \mathcal{B}_{n-3}$ and $q(G_u')=2^{-\frac{3}{5}}198^{-\frac{3}{5}}q(G)$. If $G'_u\in \mathcal{E}$, then $G'_u=K_6^-$ since $|V(G'_u)|\geq 6$. Furthermore, $G[N]$ is isomorphic to the graph obtained from $K_{6}^-$ by deleting a perfect matching, that is, $3$-prism. Since $G$ has no $D_4$ by Claim \ref{xh-7}, up to isomorphism, we deduce that $G\in \{Z_1,Z_2,Z_3\}\subseteq \mathcal{M}$, a contradiction. Now assume that $G'_u\not\in \mathcal{E}$. By Lemma \ref{x322}, $F(G')\geq\frac{27}{7}F(G_u')\geq \frac{27}{7}q(G_u')=\frac{27}{7}\cdot2^{-\frac{3}{5}}198^{-\frac{3}{5}}q(G)$. Observe that every spanning forest of $G'$ yields $3$ different spanning forests  of $G$ containing  exactly two edges in $\{xy,zx,yz\}$, thus yielding at least $\frac{81}{7}\cdot2^{-\frac{3}{5}}198^{-\frac{3}{5}}q(G)$ different spanning forests of $G$, each containing precisely $2$ edges in $\{xy,zx,yz\}$.

In total, we have constructed $$\left(
2^{12/5}198^{-3/5}
+
3\cdot 2^{3/5}198^{-2/5}
+
\frac{81}{7}\cdot 2^{-3/5}198^{-3/5}
\right)q(G)$$ different spanning forests of $G$, which is more than $q(G)$. This completes the proof.
{\hfill}
\end{proof}

\begin{claim}\label{xh-9}
 If $x\in V(G)$ with $d_G(x)=4$, then $G[N_G(x)]$ is not empty.
\end{claim}
\begin{proof}[\em\textbf {Proof of Claim \ref{xh-9}}]
Suppose that $N_G(x)=\{a,b,c,d\}$ induces no edge in $G$. Then we can perform a complete lift of $x$ in $G$, and it amounts to deleting $x$ and adding the edges of a perfect matching between its neighbours.
We consider $3$ graphs that may be formed by doing a complete lift at $x$: from $G-x$, we obtain $G_1$ by adding the edges in $M_1=\{ab,cd\}$; $G_2$ by adding the edges in $M_2=\{ac,bd\}$; and $G_3$ by adding the edges in $M_3=\{ad,bc\}$.
As $x$ is not a cut-vertex of $G$ by Claim \ref{xh-3}, we have $G_1,G_2,G_3\in \mathcal{B}_{n-1}$. Since $G[\{a,b,c,d\}]$ is empty,  $G_1,G_2,G_3\not\in \mathcal{E}$. Note that $F(G_j)\geq q(G_j)=2^{-\frac{1}{5}}198^{-\frac{1}{5}}q(G)$ for each $j\in [3]$.

Now we consider all spanning forests   of $G_j-M_j=G-x$  that can be extended to a spanning forest of $G_j$ by adding (possibly $0$) edges in $M_j$, where $j\in [3]$. We classify in Table \ref{tab2} all these spanning forests into $15$ types, regarding which vertices among $a,b,c,d$ belong to the same component, really only $5$ types up to the symmetry of the roles of these vertices. For each such type, we calculate the number of ways each corresponding spanning forest may be extended into a spanning forest of $G$, $G_1$, $G_2$ and $G_3$ by using only edges in $\{xa,xb,xc,xd\}$, $M_1$, $M_2$ and $M_3$, respectively.

For each $i\in [15]$ and $j\in [3]$, let $t_i$ be the number of spanning  forests of type $i$, let  $\lambda_i$ be the number of ways to extend a spanning forest of type $i$ to a spanning forest of $G$ and let  $\lambda_{i,j}$ be the number of ways to extend a spanning forest of type $i$ to a spanning forest of $G_j$. Then
\begin{equation}\label{v22}
\begin{split}
F(G)=\sum_{i=1}^{15}\lambda_i t_i~~~~~~{\rm and }~~~~~F(G_j)=\sum_{i=1}^{15}\lambda_{i,j}t_i,~~\forall j\in [3].
\end{split}
\end{equation}
We further observe from Table \ref{tab2} that  $\lambda_i\geq\frac{6}{5}(\lambda_{i,1}+\lambda_{i,2}+\lambda_{i,3})$ for any $i\in[15]$. Then \eqref{v22} implies that $F(G)\geq \frac{6}{5}\left(F(G_1)+F(G_2)+F(G_3)\right)$. Recall that $F(G_j)\geq q(G_j)=2^{-\frac{1}{5}}198^{-\frac{1}{5}}q(G)$ for each $j\in [3]$. Therefore, we have
$$F(G)\geq 3\cdot \frac{6}{5}\cdot2^{-\frac{1}{5}}198^{-\frac{1}{5}}q(G)> q(G),$$ a contradiction.
This completes the proof.
{\hfill}
\end{proof}

\begin{table}[htbp]
  \centering
  \begin{tabular}{lccccc} 
    \Xhline{1pt}
     Type    &Forest Type &  $G$ & $G_1$ &$G_2$ &$G_3$\\ 
    \Xhline{1pt}
    $1$    & $\{a\},\{b\},\{c\},\{d\}$     &$16$ &$4$ &4& $4$\\
     $2$   & $\{a,b\},\{c\},\{d\}$       &$12$ &$2$ &$4$&$4$\\
    $3$    & $\{a,c\},\{b\},\{d\}$      &$12$   &$4$&$2$&$4$ \\
     $4$   &  $\{a,d\},\{b\},\{c\}$    &$12$    &$4$&$4$&$2$\\
     $5$   & $\{b,c\},\{a\},\{d\}$      &$12$ &$4$ &$4$&$2$\\
       $6$   & $\{b,d\},\{a\},\{c\}$      &$12$ &$4$&$2$&$4$ \\
        $7$   & $\{c,d\},\{a\},\{b\}$      &$12$ &$2$ &$4$  &$4$\\
         $8$   & $\{a,b\},\{c,d\}$      &$9$ &$1$&$3$&$3$ \\
          $9$   & $\{a,c\},\{b,d\}$      &$9$ &$3$ &$1$& $3$\\
          $10$    & $\{a,d\},\{b,c\}$      &$9$   &$3$&$3$&$1$ \\
          $11$   & $\{a\},\{b,c,d\}$      &$8$ &$2$ &$2$&$2$ \\
        $12$   & $\{b\},\{a,c,d\}$      &$8$ &$2$  &$2$&$2$\\
         $13$   & $\{c\},\{a,b,d\}$      &$8$ &$2$&$2$&$2$ \\
          $14$   & $\{d\},\{a,b,c\}$      &$8$ &$2$ &$2$&$2$ \\
          $15$    & $\{a,b,c,d\}$      &$5$   &$1$&$1$&$1$ \\
   \Xhline{1pt}
  \end{tabular}
  \caption{All the possible types of spanning forests of $G-x=G_i-M_i$, along with the number of ways to extend them into spanning forests of $G$, $G_1$, $G_2$ and $G_3$.}
  \label{tab2}
\end{table}

Now we start to prove Theorem \ref{xh-1}.
\begin{proof}[\em\textbf {Proof of Theorem \ref{xh-1}}]
Suppose that there exists a graph $G\in \mathcal{B}_{n}\setminus \mathcal{E}$  with the minimum number of edges  such that $F(G)<q(G)$. Note that $n\geq4$ and $G\neq K_4$ as $F(K_4)=38> 37=\lceil q(K_4)\rceil$. If $G$ has maximum degree less than $4$, then Theorem \ref{x2} implies that
\begin{displaymath}
\begin{split}
F(G)&\geq2^{n_2+n_3-1}3^{\frac{n_3+2}{4}}=2^{n_2-3}\times(2\times3^{\frac{1}{4}})^{n_3+2}\\
&>2^{n_2-3}\times(2^{\frac{3}{5}}198^{\frac{1}{10}})^{n_3+2}\\
&=2^{n_2+\frac{3n_3-9}{5}}198^{\frac{n_3+2}{10}}=q(G)
\end{split}
\end{displaymath}
since $2\times3^{1/4}>2^{3/5}198^{1/10}$, a contradiction. Then $G$ has a vertex $v$ with $d_G(v)=4$. By Claim \ref{xh-9}, there are $2$ neighbours $u$ and $w$ of $v$ satisfying $uw\in E(G)$, that is, $\{u,v,w\}$ induces a triangle in $G$, which contradicts Claim \ref{xh-8}. This completes the proof.
{\hfill}
\end{proof}

\section*{Acknowledgements}
The work was supported by National Natural Science Foundation of China (Grant No.\ 12271251), Postgraduate Research \& Practice Innovation Program of Jiangsu Province, grant number KYCX25\_0625.


\section*{Appendix}
Link to SageMath code for checking some calculations in this paper: \url{https://doi.org/10.5281/zenodo.17339883}.

\end{document}